\documentclass[11pt]{amsart}
\usepackage{amssymb,amsmath}
\usepackage{amsfonts}
\usepackage[T1]{fontenc}

\newtheorem{theorem}{Theorem}[section]
\newtheorem{lemma}[theorem]{Lemma}
\newtheorem{proposition}[theorem]{Proposition}
\newtheorem{corollary}[theorem]{Corollary}

\newtheorem{notation}[theorem]{Notation}
\newtheorem{fact}[theorem]{Fact}
\newtheorem*{claim}{Claim}
\theoremstyle{definition}
\newtheorem{definition}[theorem]{Definition}

\theoremstyle{remark}
\newtheorem{remark}[theorem]{Remark}

\newcommand{\bl}{\begin{lemma}}
\newcommand{\el}{\end{lemma}}
\newcommand{\bfa}{\begin{fact}}
\newcommand{\efa}{\end{fact}}
\newcommand{\bpr}{\begin{proposition}}
\newcommand{\epr}{\end{proposition}}
\newcommand{\bp}{\begin{proof}}
\newcommand{\ep}{\end{proof}}
\newcommand{\bd}{\begin{definition}}
\newcommand{\ed}{\end{definition}}
\newcommand{\bt}{\begin{theorem}}
\newcommand{\et}{\end{theorem}}
\newcommand{\bc}{\begin{corollary}}
\newcommand{\ec}{\end{corollary}}
\newcommand{\bn}{\begin{notation}}
\newcommand{\en}{\end{notation}}
\newcommand{\br}{\begin{remark}}
\newcommand{\er}{\end{remark}}
\newcommand{\bcl}{\begin{claim}}
\newcommand{\ecl}{\end{claim}}

\newcommand{\N}{{\mathbb{N}}}
\newcommand{\R}{{\mathbb{R}}}

\newcommand{\abs}[1][\cdot]{\lvert#1\rvert}
\newcommand{\nrm}[1]{\|#1\|}
\newcommand{\al}{\alpha}
\newcommand{\e}{\varepsilon}
\newcommand{\de}{\delta}
\newcommand{\bnum}{\begin{enumerate}}
\newcommand{\enum}{\end{enumerate}}
\newcommand{\mc}{\mathcal}
\newcommand{\mt}{\mc{T}}

\newcommand{\fa}{f_{\alpha}}
\newcommand{\fb}{f_{\beta}}

\numberwithin{subsection}{section}
\numberwithin{equation}{section}

\newcommand{\norm}[1][\cdot]{\lVert #1\rVert}
\newcommand{\xn}[1][x]{(#1_{n})_{n\in\N}}
\DeclareMathOperator{\supp}{supp}
\DeclareMathOperator{\maxsupp}{maxsupp}
\DeclareMathOperator{\minsupp}{minsupp}
\DeclareMathOperator{\ran}{range}

\DeclareMathOperator{\suc}{succ}

\DeclareMathOperator{\ord}{ord}
\begin{document}
\title{Isomorphisms and strictly singular operators \\ in mixed Tsirelson spaces}
\author{Denka Kutzarova} 
\address{Institute of Mathematics, Bulgarian Academy of Sciences, current address: Department of Mathematics, University of Illinois at
Urbana-Champaign}
\email{denka@math.uiuc.edu}
\author{Antonis Manoussakis}
\address{Department of Sciences, Technical University of Crete, 
73100 Chania (Crete), Greece} 
\email{amanousakis@isc.tuc.gr}
\author{Anna  Pelczar-Barwacz} 
\address{Institute of Mathematics, Jagiellonian University, {\L}ojasiewicza 6, 30-348 Krak\'ow, Poland}
\email{anna.pelczar@im.uj.edu.pl}
\thanks{The research of the third author was supported by the Polish Ministry of Science and Higher Education grant N N201 421739}
\begin{abstract}
We study the family of isomorphisms and strictly singular operators in mixed Tsirelson spaces and their modified versions setting. We show sequential minimality of modified mixed Tsirelson spaces $T_M[(\mc{S}_n,\theta_n)]$ satisfying some regularity conditions and present results on existence of strictly singular non-compact operators on subspaces of mixed Tsirelson spaces defined by the families $(\mc{A}_n)_n$ and $(\mc{S}_n)_n$. 
\end{abstract}
\keywords{quasiminimality, strictly singular operator, mixed Tsirelson space }
\maketitle
\section*{Introduction}
In the celebrated paper \cite{G} W.T. Gowers started his classification program for  Banach spaces. The goal is to identify classes of Banach spaces which are 
\bnum
 \item hereditary, i.e. if a space belongs to a given class, then all of its closed infinite dimensional subspaces as well,  
 \item inevitable, i.e. any Banach space contains an infinite dimensional subspace in one of those classes, 
 \item defined in terms of richness of family of bounded operators in the space. 
\enum
The famous Gowers' dichotomy brought first two classes: spaces with unconditional basis and hereditary indecomposable spaces. The further classification, described in terms of isomorphisms, concerned minimality and strict quasiminimality. A Banach space $X$ is \textit{ minimal} if every closed infinite dimensional subspace of  $X$ contains a further subspace  isomorphic to $X$.  A  Banach space $X$ is called \textit{quasiminimal} if any two infinite dimensional subspaces $Y,Z$ of $X$ contain further isomorphic subspaces.  The classical spaces $\ell_{p}$, $1\leq p<\infty$, $c_{0}$ are minimal and the Tsirelson space $T[\mc{S}_1,1/2]$ is the first known strictly quasiminimal space (i.e. without minimal subspaces), \cite{CO}.  The results of W.T. Gowers lead to the question of the refinement of the classes and classification of already  known  Banach space. Further step in the first direction was made by the third named author, \cite{p}, who proved that a strictly quasiminimal  Banach space contains a subspace with no subsymmetric sequence. An extensive refinement of list of the classes and study of exampes were made recently by V. Ferenczi and C. Rosendal \cite{fr, fr2}.

The mixed Tsirelson spaces $T[(\mathcal{M}_{n},\theta_{n})_n]$, for $\mc{M}_n=\mc{A}_n$ or $\mc{S}_n$, as the basic examples of spaces not containing $\ell_p$ or $c_0$, form a natural class to be studied with respect to the classification program. The first step was made by  T. Schlum\-precht, \cite{as1}, who proved that his famous space $S=T[(\mathcal{A}_{n},1/\log_{2}(n+1))_{n}]$  is complementably minimal. The result  of  Schlumprecht holds for a certain class of mixed Tsirelson spaces $T[(\mathcal{A}_{k_n},\theta_n)_{n}]$ by \cite{m}. On the other hand, the  Tzafriri's  space $T[(\mathcal{A}_{n},c/\sqrt{n})_{n}]$ \cite{t} is not minimal by \cite{jko}. However the original Tsirelson space $T[\mc{S}_1,1/2]$ is not minimal \cite{CO}, every its normalized block sequence is equivalent to a subsequence of the basis. We show that mixed Tsirelson spaces $T[(\mathcal{A}_{n},\theta_n)_{n}]$, for which Tzafriri space is a prototype, are saturated with subspaces with this "blocking principle''.

V. Ferenczi and C. Rosendal \cite{fr} introduced and studied a stronger notion of quasiminimality. A Banach space $X$ with a basis is \textit{sequentially minimal} \cite{fr}, if any block subspace of $X$ contains a block sequence $(x_n)$ such that every block subspace of $X$ contains a copy of a subsequence of $(x_n)$. The related notions  in mixed Tsirelson spaces defined by families $(\mc{S}_n)$ and their relation to existence of $\ell_1^\omega$-spreading models were studied in \cite{lt,klmt}. In \cite{mp} it was shown that the spaces $T[(\mathcal{A}_{n},\theta_n)_{n}]$, as well as $T[(\mathcal{S}_{n},\theta_n)_n]$ satisfying the regularity condition $\theta_n/\theta^n\searrow$, where $\theta=\lim_n\theta_n^{1/n}$, are sequentially minimal. We show that the  modified mixed Tsirelson spaces  $T_{M}[(\mathcal{S}_{n},\theta_{n})_{n}]$ with the above property are also sequentially minimal. 

The major tool  in the study of mixed Tsirelson spaces $T[(\mc{S}_n,\theta_n)_n]$ are the tree-analysis of norming functionals and the special averages introduced in \cite{ad2}, see also \cite{ato}. The basic idea to prove quasiminimality is to produce in every subspace a sequence of appropriate special averages of rapidly increasing lengths and show these sequences span isomorphic subspaces. 
The major obstacle in study of modified mixed Tsirelson spaces is estimating the norms of splitting a vector into pairwise disjoint parts instead of consecutive parts as in non-modified setting. In order to overcome it, we introduced special types of averages, so-called Tsirelson averages, describing in fact local representation of the Tsirelson space $T[\mc{S}_1,\theta]$, with $\theta=\sup_n\theta_n^{1/n}$, in the considered space. Then we are able to control the action of a norming functional on a linear combination of Tsirelson averages by the action of a norming functional on suitable averages in the Tsirelson space $T[\mc{S}_1,\theta]$ and vice versa. Using those estimations we prove the sequential minimality of modified mixed Tsirelson space satisfying the regularity condition. Tsirelson averages are  also  the main tool for proving arbitrary distortability of $T_M[(\mc{S}_n,\theta_n)]$ in case $\theta_n/\theta^n\searrow 0$, the result known before in non-modified setting under the condition $\theta_n/\theta^n\to 0$, \cite{ao}.

In the second part of the paper we deal with the existence of strictly singular  non-compact operators in mixed Tsirelson spaces.  The existence of non-trivial strictly singular operators, i.e. operators whose none restriction to an infinite dimensional subspace is an isomorphism, was also studied in context of classification program of Banach space, both in search for sufficient conditions and examples on known spaces.  A space on which all the bounded operators are compact perturbations of  multiple of the identity was constructed recently by S.A. Argyros and R. Haydon, \cite{ah}, who solved "scalar-plus-compact". The existence of strictly singular non-compact operators was shown on Gowers-Maurey spaces and Schlumprecht space \cite{as2}, as well as on a class of spaces defined by families $(\mc{S}_n)_n$ \cite{gas2}. Th. Schlumprecht \cite{s2} studying the richness of the family of operators on a Banach space in connection with the "scalar-plus-compact" problem defined two classes of Banach spaces. Class 1 refers to a variation of a "blocking principle'', while Class 2 means existence of a striclty singular non-compact operator in any subspace (see Def. \ref{s2}). T. Schlumprecht asked if any Banach space contains a subspace with a basis which is either of Class 1 or Class 2. We show that a mixed Tsirelson space  $T[(\mathcal{A}_{n}, \frac{c_{n}}{n^{1/q}})_{n}]$  belongs to Class 1  if   $\inf_{n}c_{n}>0$   and   to Class  2  if   $\lim_{n}c_{n}=0$.

In \cite{kl}  a block sequence $(x_{n})_{n\in\mathbb{N}}$ generating  $\ell_{1}$-spreading model was constructed in Schlum\-precht space $S$. This result combined with the result  of   I. Gasparis  \cite{gas2} led to the question if some biorthogonal sequence to $(x_{n})_{n}$ generates a $c_{0}$-spreading model in $S^{*}$.  We remark that this is not the case. In general, it is still unknown if any sequence in $S^*$ generates a $c_0$-spreading model. Finally we show that in mixed (modified) Tsirelson spaces defined by $(\mc{S}_n)$  containing a block sequence generating  $\ell_{1}^{\omega}$-spreading model there is a strictly singular non-compact operator on a subspace.

We describe now briefly the content of the paper. In the first section we recall the basic notions in the theory of mixed Tsirelon spaces and their modified versions, including the canonical representation of these spaces and the notion of a tree-analysis of  a norming functional (Def. \ref{def-tree}). The second section is devoted to the study of modified mixed Tsirelson spaces $T[(\mc{S}_n,\theta_n)_n]$ satisfying the regularity condition. We extend the notion of an averaging tree (Def. \ref{aat0}) and present the notions of averages of different types, providing also upper (Lemma \ref{theta1.7}) and lower (Lemma \ref{theta3}) "Tsirelon-type" estimates.  We conclude the section with the result on arbitrary distortion for spaces with $\theta_n/\theta^n\searrow 0$ (Theorem \ref{dist}) and sequential minimality (Theorem \ref{quasi}). In the last section we study the existence of non-compact strictly singular operators in mixed Tsirelson spaces $T[(\mc{A}_n,\theta_n)_n]$ (Theorem \ref{p-sp}). We discuss the behaviour of a biorthogonal sequence to the sequence generating $\ell_1$-spreading model in Schlumprecht space (Proposition \ref{c-0}) and the case of mixed Tsirelson spaces $T[(\mc{S}_n,\theta_n)_n]$ admitting $\ell_1^\omega$-spreading model (Theorem \ref{thm3.7}). We finish with the comments and questions concerning the Tzafriri space and richness of the set of subsymmetric sequences in a Banach space. 

\section{Preliminaries}
We recall  the basic definitions and standard notation.

By a {\em  tree}  we shall mean a non-empty partially ordered  set $(\mt, \preceq)$ for which the set $\{ y \in \mt:y \preceq x \}$ is linearly ordered and finite for each $x \in \mt$. If $\mt' \subseteq \mt$ then we say that $(\mt', \preceq)$ is a {\em subtree}  of $(\mt,\preceq)$. The tree $\mt$ is called {\em finite}  if the set $\mt$ is finite. The {\em initial} nodes of $\mt$ are the minimal elements of $\mt$ and the {\em terminal}  nodes are the maximal elements. A {\em branch}  in $\mt$ is a maximal linearly ordered set in $\mt$. The {\em immediate successors}  of $x \in \mt$, denoted by $\succ (x)$, are all  the nodes $y \in \mt$ such that $x \preceq y$ but there is no $z \in \mt$ with $x \preceq z \preceq y$. If $X$ is a linear space, then a {\em tree in $X$}  is a tree whose nodes are vectors in $X$.

Let $X$ be a Banach space with a basis $(e_i)$. The \textit{support} of a vector $x=\sum_{i} x_i e_i$ is the set $\supp x =\{ i\in \N : x_i\neq 0\}$, the \textit{range} of $x$, denoted by $\ran (x)$ is the minimal interval containing $\supp x$. Given any $x=\sum_{i} a_ie_i$ and finite $E\subset\N$ put $Ex=x_E=\sum_{i\in E}a_ie_i$. We write $x<y$ for vectors $x,y\in X$, if $\max\supp x<\min \supp y$. A \textit{block sequence} is any sequence $(x_i)\subset X$ satisfying $x_{1}<x_{2}<\dots$, a \textit{block subspace} of $X$ - any closed subspace spanned by an infinite block sequence. A subspace spanned by a block sequence $(x_n)$ we denote by $[x_n]$. 
\bn 
Given any two vectors $x,y\in X$ we write $x\preceq y$, if $\supp x\subset\supp y$, and we say that $x$ and $y$ are \textit{incomparable}, if $\supp x\cap \supp y=\emptyset$. 

Given a block sequence $(x_n)\subset X$ and a functional $f\in X^*$ we say that $f$ begins in $x_{n}$, if $\minsupp f\in (\maxsupp x_{n-1},\maxsupp x_n]$ (set $x_0=0$).
\en
A basic sequence $(x_n)$ $C-$\textit{dominates} a basic sequence $(y_n)$, $C\geq 1$, if for any scalars $(a_n)$ we have
$$
\nrm{\sum_{n}a_ny_n}\leq C\nrm{\sum_{n}a_nx_n}\,.
$$
Two basic sequences $(x_n)$ and $(y_n)$ are $C$-\textit{equivalent}, $C\geq 1$, if $(x_n)$ $C-$dominates $(y_n)$ and $(y_n)$ $C-$dominates $(x_n)$.
\bd
Let $E$ be a Banach space with a 1-subsymmetric basis $(u_n)$, i.e. 1-equivalent to any of its infinite subsequences. Let $(x_n)$ be a seminormalized basic sequence in a Banach space $X$. We say that $(x_n)_n$ generates $(u_n)$ as a spreading model, if for any $k\in\N$ and any $(a_i)_{i=1}^k\subset\R$ we  have
$$
\lim_{n_1\to\infty}\lim_{n_2\to\infty}\dots\lim_{n_k\to\infty}\norm[\sum_{i=1}^ka_ix_{n_i}]_X=\norm[\sum_{i=1}^ka_iu_i]_E\,.
$$
We say that a Banach space $X$ with a basis is $\ell_p$-asymptotic, $1\leq p\leq\infty$, if any block sequence $(x_i)_{i=1}^n$ is $C$-equivalent to the u.v.b. of $\ell_p^n$, for some universal $C\geq 1$. 
\ed
By \cite{bs} any seminormalized basic sequence admits a subsequence generating spreading model. We say that $(x_n)$ generates $\ell_p$- (resp. $c_0$-)spreading model, if $(u_n)$ is equivalent to the u.v.b. of $\ell_1$ (resp. $c_0$). 

Recall that by Krivine theorem for any Banach space $X$ with a basis there is some $1\leq p\leq \infty$ such that $\ell_p$ is finitely block (almost isometrically) represented in $X$, i.e. for any $\e>0$ and any $n\in\N$ there is a normalized block sequence $x_1<\dots<x_n$ in $X$ which is $(1+\e)$-equivalent to the u.v.b. of $\ell_p^n$. 

\

We work on two types of families of finite subsets of $\N$: $(\mc{A}_n)_{n\in\N}$ and $(\mc{S}_\al)_{\al<\omega_1}$. Let
$$
\mc{A}_n=\{F\subset\N:\# F\leq n\}, \ \ n\in\N\,.
$$
\textit{Schreier families} $(\mc{S}_\al)_{\al<\omega_1}$, introduced in \cite{aa}, are defined by induction:
\begin{align*}
\mc{S}_0 &=\{\{ k\}:\ k\in\N\}\cup\{\emptyset\}, \\
\mc{S}_{\al+1}&  =\{F_1\cup\dots\cup F_k:\ k\leq F_1<\dots<F_k, \
f_1,\dots, F_k\in \mc{S}_\al\}, \ \ \al<\omega_1\,.
\end{align*}
If $\al$ is a limit ordinal, choose $\al_n\nearrow \al$ and set 
$$
\mc{S}_\al=\{F:\ F\in \mc{S}_{\al_n}\ \mathrm{and}\ n\leq F\ \mathrm{for\ some}\ n\in\N\}\,.
$$
Given a family $\mc{M}=\mc{A}_n$ or $\mc{S}_n$ we say that a sequence $E_1,\dots, E_k$ of subsets of $\N$ is
 \bnum
 \item $\mc{M}$-\textit{admissible}, if $E_1<\dots<E_k$ and $(\min E_i)_{i=1}^k\in\mc{M}$,
 \item $\mc{M}$-\textit{allowable}, if $(E_i)_{i=1}^k$ are pairwise disjoint and $(\min E_i)_{i=1}^k\in\mc{M}$.
 \enum
Let $X$ be a Banach space with a basis. We say that a sequence $x_1<\dots <x_n$ is $\mc{M}$-\textit{admissible} (resp. \textit{allowable}), if $(\supp x_i)_{i=1}^n$ is $\mc{M}$-admissible (resp. allowable).

\bd[Mixed and modified mixed Tsirelson space] Fix a sequence of families $(\mc{M}_n)=(\mc{A}_{k_n})$ or $(\mc{S}_{k_n})$ and sequence $(\theta_n)\subset (0,1)$ with $\lim_{n\to\infty}\theta_n=0$. Let $K\subset c_{00}$ be the smallest set satisfying the following:
\bnum 
 \item $(\pm e_n^*)_n\subset K$,
 \item for any $f_1<\dots<f_k$ in $K$, if $(f_i)_{i=1}^k$ is $\mc{M}_n$-admissible for some $n\in\N$, then $\theta_n(f_1+\dots+f_k)\in K$.
\enum 
We define a norm on $c_{00}$ by $\nrm{x}=\sup\{f(x):f\in K\}$, $x\in c_{00}$. The \textit{mixed Tsirelson space} $T[(\mc{M}_n,\theta_n)_n]$ is the completion of $(c_{00}, \nrm{\cdot})$.

The \textit{modified mixed Tsirelson space} $T_M[(\mc{M}_n,\theta_n)_n]$ is defined analogously, by replacing admissibility by allowability of the sequences.
\ed
It is standard to verify that the norm $\nrm{\cdot}$  is the unique norm on $c_{00}$ satisfying the equation
$$
\nrm{x}=\max\left\{\nrm{x}_\infty,\sup\left\{\theta_n\sum_{i=1}^k\nrm{E_ix}: \ (E_i)_{i=1}^k - \mc{M}_n- \mathrm{admissible}, \ n\in\N\right\}\right\}\,.
$$
It follows immediately that the u.v.b. $(e_n)$ is 1-unconditional in the space $T[(\mc{M}_n,\theta_n)_n]$. It was proved in \cite{ad2} that any $T[(\mc{S}_{k_n},\theta_n)_n]$ is reflexive, also any $T[(\mc{A}_{k_n},\theta_n)_n]$ is reflexive, provided $\theta_n>\frac{1}{k_n}$ for at least one $n\in\N$, \cite{ato}. 

Taking $\mc{M}_n=\mc{M}$ and $\theta_n=\theta$ for any $n$  we obtain the classical Tsirelson-type space $T[\mc{M},\theta]$. Recall that  $T[\mc{A}_n,\theta]=c_0$ if $\theta\leq 1/n$ and $T[\mc{A}_n,\theta]=\ell_p$, if $\theta=1/\sqrt[q]{n}$ for $q$ satisfying $1/p+1/q=1$, \cite{b,ato}. The space $T[\mc{S}_1,1/2]$ is the Tsirelson space.

Schlumprecht space $S$ is the space $T[(\mc{A}_n,\frac{1}{\log_2(n+1)})_n]$, Tzafriri space is $T[(\mc{A}_n,\frac{c}{\sqrt{n}})_n]$ for $0<c<1$. Modified Tsirelson-type spaces are isomorphic to their non-modified version, whereas the situation is quite different in mixed setting, \cite{adkm}.

We present now the canonical form of (modified) mixed Tsirelson space in both cases $\mc{M}_n=\mc{A}_{k_n}$ or $\mc{S}_{k_n}$, $n\in\N$.
\bd \cite{m} A mixed Tsirelson space $T[(\mc{A}_{k_n},\theta_n)_{n\in\N}]$ is called a \textit{$p-$space}, for $p\in [1,\infty)$, if there is a sequence $(p_N)_N\subset (1,\infty)$ such that
\bnum 
 \item $p_N\to p$ as $N\to\infty$, and $p_N\geq p_{N+1}>p$ for any $N\in\N$,
 \item $T[(\mc{A}_{k_n},\theta_n)_{n=1}^N]$ is isomorphic to $\ell_{p_N}$ for any $N\in\N$.
\enum
\ed
A $p-$space $T[(\mc{A}_n,\theta_n)_{n\in\N}]$ is called \textit{regular}, if
$\theta_n\searrow 0$ and $\theta_{nm}\geq\theta_n\theta_m$ for any $n,m\in\N$. Recall that any $p-$space is isometric to a regular $p-$space \cite{mp}.
\begin{notation}
Let $T[(\mc{A}_n,\theta_n)_{n\in\N}]$ be a regular $p-$space. If we set
$\theta_n=1/n^{1/q_n}$ with $q_n\in (1,\infty)$, $n\in\N$, then $q=\lim_{n} q_n=\sup_n
q_n\in (0,\infty]$, where $1/p+1/q=1$, with usual convention $1/\infty=0$.

In the situation as above let $c_n=\theta_nn^{1/q}\in (0,1)$, $n\in\N$, if $p>1$. To
unify the notation put $c_n=\theta_n$, $n\in\N$, in case $p=1$.
\end{notation}
A space $T_M[(\mc{S}_n,\theta_n)_{n\in\N}]$ with $\theta_n\searrow 0$ and $\theta_{n+m}\geq\theta_n\theta_m$  is called a \textit{regular} space. 
Notice that any modified mixed Tsirelson space is isometric to a regular modified mixed Tsirelson space (cf. \cite{ao}). 
\bn For a regular modified mixed Tsirelson space $T_M[(\mc{S}_n,\theta_n)_n]$ let $\theta=\lim_{n}\theta_n^{1/n}=\sup_n\theta_n^{1/n}\in (0,1]$.
We shall use also the following condition: 
$$
(\clubsuit)\ \ \ \ \ (\theta_n/\theta^n)_n\searrow\ \  \text{ i.e. } \ \ \theta_{n+m}\leq \theta_n\theta^m  \text{ for any } n,m\in\N. 
$$
\en
\begin{lemma}\label{x2}
  The space  $T_{M}[(\mc{S}_{n}[A_{2}],\theta_{n})_{n}]$ is $3$-isomorphic to  $T_{M}[(\mc{S}_{n},\theta_{n})_{n}]$.
\end{lemma}
The proof of the above follows  that  of Lemma 4.5, \cite{mp} with  "admissible" sequences replaced by "allowable" ones.

The following notion provides a useful tool for estimating norms in Tsirelson type spaces, mixed Tsirelson spaces and their modified versions:
\bd\label{def-tree}[The tree-analysis of a norming functional] Let $f\in K$, the norming set of $T[(\mc{M}_n,\theta_n)_n]$ (resp.  $T_M[(\mc{M}_n,\theta_n)_n]$). By a \textit{tree-analysis} of $f$ we mean a finite family $(f_\al)_{\al\in \mt}$ indexed by a tree $\mt$ with a unique root $0\in \mt$ (the smallest element) such that the following hold
\bnum
 \item $f_0=f$ and $f_\al\in K$ for all $\al\in \mt$,
 \item $\al\in T$ is maximal if and only if $f_\al\in (\pm e_n^*)$,
 \item for every not maximal $\al\in T$ there is some $n\in\N$ such that $(f_\beta)_{\beta\in\suc (\al)}$ is an $\mc{M}_n$-admissible (resp. -allowable) sequence and $f_\al=\theta_n(\sum_{\beta\in\suc (\al)}f_\beta)$. We call $\theta_n$ the \textit{weight} of $f_\al$. 
 \enum
For any $\al\in \mt$, $\al>0$, we define the tag $t(\al)=t(f_\al)$ as $t(\al)=\prod_{\al>\beta\geq 0}weight(f_\beta)$. 

For any $\al\in\mt$we define also inductively the order of $\al$ as follows: $\ord(0)=0$ and for any $\beta\in\suc (\al)$ we put $\ord(\beta)=\ord(\al)+n$, where $weight(f_\al)=\theta_n$.
 \ed
Notice that every functional $f\in K$ admits a tree-analysis, not necessarily unique.

We shall use repeatedly  the following
\begin{fact} \label{f2} Let $X=T_M[(\mc{S}_n,\theta_n)_n]$ with $(\clubsuit)$. Let  $(f_{\alpha})_{\alpha\in\mc{T}}$ be a norming tree of a norming functional $f\in K$ and  $\alpha$ not a terminal node.  Let $\fa=\theta_{r_{\alpha}}\sum_{\beta\in  succ(\alpha)}\fb$. Then for every $k\in[\ord(\alpha),\ord(\alpha)+r_{\alpha}]$ we get
\begin{equation*}
\fa=\theta_{r_{\alpha}}\sum_{t\in A_{\alpha}}\sum_{s\in F_{t}}f_{s}
\end{equation*}
where $(f_{s})_{s\in F_{t}}$  is $\mc{S}_{r_{\alpha}-(k-\ord(\alpha))} $-allowable, for any $t\in A_{\alpha}$, and  $(g_{t})_{t\in A_{\alpha}}$ is $\mc{S}_{k-\ord(\al)}$-allowable, for $g_{t}=\theta_{r_{\alpha}-(k-\ord(\alpha))}\sum_{s\in F_{t}}f_{t}$, $t\in A_\al$. In particular by $(\clubsuit)$ we get 
\begin{equation*}
\fa(x)\leq \theta^{k-\ord(\alpha)}\sum_{t\in  A_{\alpha}}g_{t}(x).
\end{equation*}
Moreover using that $t(\alpha)\leq\theta_{\ord(\al)}\leq\theta^{\ord(\al)}$ we have $t(\alpha)\fa(x)\leq \theta^{k}\sum_{t\in  A_{\alpha}}g_{t}(x).$
\end{fact}
\section{Modified mixed Tsirelson spaces defined on Schreier families}
In this section we present the main results on sequential minimality and arbitrary distortability of a regular modified mixed Tsirelson spaces $T_M[(\mc{S}_n,\theta_n)]$ with $(\clubsuit)$. In the first subsection we discuss the notions of averages of different types, in the next two subsections we present estimations on their norms. Since the u.v.b. in any (modified) mixed Tsirelson space and its dual is unconditional, we work in the sequel on functionals and vectors with non-negative coefficients.

\subsection{Averages}
In this part we present the notion of special averages and recall basic facts. Let $X$ be a Banach space with a basis. We  will use a version of the notion of special averages introduced in \cite{ad2}.
\bd A vector $x\in X$ is called an $(M,\e)$-average of a block sequence $(x_i)_i\subset X$, for $M\in\N$ and $\e>0$, if $x=\sum_{i\in G}a_ix_i$ for some $G\in\mc{S}_M$ and $(a_i)_{i\in G}\subset (0,1]$ with $\sum_{i\in G}a_i=1$ and for any $F\in\mc{S}_{M-1}$ we have $\sum_{i\in F}a_i<\e$. 
\ed
We use the notion of an averaging admissible tree,  \cite{ao}, with additional features:
\bd\label{aat0} We call a tree $(x_i^j)_{j=0,i=1}^{M,N^j}$ in $X$ with weights $(N_i^j)_{j=1,i=1}^{M,N^j}\subset\N$ and errors $(\e_i^j)_{j=1,i=1}^{M,N^j}\subset (0,1)$, an averaging tree, if
\bnum
\item $(x_i^j)_{i\in I_j}$ is a block sequence for any $j$, $1=N^M\leq\dots\leq N^0$.

\noindent Moreover for any $j=1,\dots,M$ and $i=1,\dots,N^j$ we have the following
\item there exists a nonempty interval $I_i^j\subset\{1,\dots,N^{j-1}\}$ with $\# I_i^j=N_i^j$ such that
 $\suc (x_i^j)=(x_s^{j-1})_{s\in I_i^j}$,
\item $x_i^j=1/N_i^j\sum_{s\in I_i^j}x_s^{j-1}$,
\item $2/\e_i^j<N_i^j\leq\minsupp x_i^j$,
\item $\e_{i+1}^j< 1/(2^i\maxsupp x_{i}^j)$, $\maxsupp x_i^j<
N_{i+1}^j$. \enum
\ed
\br\label{aat1}
 In the situation as above we define coefficients $(a_i^j)_{j=0,i=1}^{M,N^j}\subset (0,1]$, as satisfying $x^M=\sum_{i=1}^{N^j}a_i^jx_i^j$. It follows straightforward that for any $j=0,\dots,M$, $i=1,\dots,N^j$ we have the following
 \bnum
 \item[(6)] $\sum_{i=1}^{N^j}a_i^j=1$,
 \item[(7)] $a_i^j=\prod_{r={j+1}}^M\frac{1}{N_{i_r}^r}$, where $x_{i_r}^r\succeq x_i^j$ for each $M\geq r> j$,
 \item[(8)] $a_i^j=\sum_{m: \ x_m^0\preceq x_i^j}a_m^0$.
\enum
 Notice that any $x_i^j$ is a $(j,\e_i^j)$-average of $(x_m^0)_{x_m^0\preceq x_i^j}$.
\er 
\bp To show the last statement notice that by $(4)$ for any $j,i\geq 1$ the block sequence $\suc (x_i^j)$ is $\mc{S}_1$-admissible, thus any block sequence $(x_m^0)_{x_m^0\preceq x_i^j}$ is $\mc{S}_j$-admissible. To complete the proof notice that by the standard reasoning (cf for example \cite{otw}, last part of the proof of Proposition 3.6) we have the following fact:

\noindent\textbf{Fact} Fix a block sequence $(x_m)_m$ and let $(x_i)_{i=1}^N$ be a block sequence of $(M-1,\e_{i})$-averages of $(x_m)_{m\in A_i}$ such that $N>2/\e$ and $\e_{i+1}<1/2^i\maxsupp x_i$. Then $x=\frac{1}{N}(x_{1}+\dots+x_{N})$ is a  $(M,\e)$-average of $(x_m)_{m\in A_i, i =1,\dots,N}$.
\ep
The above Lemma together with the construction of an averaging tree presented in \cite{ao} yields the standard
\bfa For any block sequence $(x_m)_m$ of $X$, any $\e>0$ and any $M\in\N$ there is an $(M,\e)$-average $x$ of $(x_m)$. 
\efa
From now on we fix a regular modified mixed Tsirelson space $X=T_M[(\mc{S}_n,\theta_n)]$. We shall use the following facts in the sequel.
\begin{fact}\label{f1}\cite{adm}
Let $x=\sum_{i\in F}a_{i}x_{i}$ be an  $(M,\e)$-average of normalized vectors $(x_i)_{i\in F}$, $M\in\N$, $\e>0$ and $\mc{E}$ an $\mc{S}_{M-1}$ allowable family of sets.  Then there is some $G\subset F$ such that for every $i\in G$ the set $\{ Ex_{i}:E\in\mc{E}, Ex_{i}\ne 0\}\,\,\, \mbox{is    $\mc{S}_{1}$-allowable}$ and
$$
\sum_{E\in\mc{E}}\norm[Ex]\leq\sum_{E\in\mc{F}}\norm[E(\sum_{i\in G}a_ix_i)]+2\e/\theta_M\,.
$$
\end{fact}
\begin{fact}\label{f3} Let $x=\sum_{i\in F}a_{i}x_{i}$ be an  $(M,\e)$-average of normalized vectors $(x_i)_{i\in F}$, $M\in\N$, $\e>0$ and $f$ a norming functional with a tree-analysis $(f_\al)_{\al\in\mt}$. Then there is subtree $\mt'$ of $\mt$ such that any terminal node of $\mt'$ has order at least $M$ and the functional $f'$ defined by the tree-analysis $(f_\al)_{\al\in\mt'}$ satisfies $f(x)\leq f'(x)+2\e$.
\end{fact}
\bp
Let $\mc{E}$ be the collection of all terminal nodes of $\mt$ of order smaller than $M$.  Let $G=\{i\in F:\ \text{ some }f_\al \text{ begins in }x_i,\ \al\in\mc{E}\}$. Since the set $(f_\al)_{\al\in\mc{E}}$ is $\mc{S}_{M-1}$-allowable, it follows $G\setminus \{\min G\}\in\mc{S}_{M-1}$ and $f(\sum_{i\in G}a_ix_i)\leq a_{\min G}+\sum_{i\in G\setminus \{\min G\}}a_i\leq 2\e$. We let $\mt'$ be the tree $\mt$ with removed nodes from the family $\mc{E}$. Then $f(x)\leq f'(x)+f(\sum_{i\in G}a_ix_i)\leq f'(x)+2\e$. 
\ep
\subsection{General estimations} 
We are able to control the norm of splitting a vector into allowable, not only admissible parts, by comparing it to the norm of splitting of a corresponding vector in the original Tsirelson space $T[\mc{S}_1,\theta]$. In this section we present the upper "Tsirelson-type" estimate for usual $(M,\e)$-averages. 

For the rest of chapter we assume that the considered regular modified mixed Tsirelson space $X=T_M[(\mc{S}_n,\theta_n)_n]$ satisfies $(\clubsuit)$. 
First we present a classical fact.
\bl\label{theta1} Let $x=\sum_{i}a_ix_i$ be an $(M,\e)$-average  of a normalized block sequence $(x_i)_i\subset X$, $M\in\N$. Then for any $j\in\N$, $j<M$ and $\mc{S}_j$-allowable $(E_l)_l$ we have
$$
\sum_l\norm[E_lx]\leq \theta^{-1}_1\theta^{M-j-1}\sum_l\sum_{i}a_i\norm[E_lx_i]+4\e/\theta_M\,.
$$
In particular $\norm[x]\leq \theta^{-1}_1\theta^{M-1}+4\e/\theta_M$. \el
\bp Take an $\mc{S}_j$-allowable sequence $(E_l)_l$. For any $l$ take a norming functional $f_l$ with $\norm[E_lx]=f_l(x)$ and its tree-analysis $(f^l_{\al})_{\al\in\mt_l}$. Let $\mc{E}$ be the collection of all terminal nodes $\al\in\mt_l$ for all $l$, such that $\ord_{\mt_l} (\al)\leq M-1-j$. Then the set $(f_\al)_{\al\in\mc{E}}$ is $\mc{S}_{M-1}$-allowable. By Fact \ref{f3}  we can assume with error $2\e$ that all terminal nodes of all $\mt_l$ have order at least $M-j$.

We will add in the tree-analysis $(f^l_\al)_{\al\in\mt_l}$'s additional nodes $(h_t)_t$ of order $M-j-1$, by grouping some of nodes of $\mt_l$, and by $(\clubsuit)$ obtain the desired estimation.

For any $l$ let $\mc{E}_l$ be collection of all $\al\in\mt_l$ which are maximal with respect to the property $\ord_{\mt_l}(\al)\leq M-j-1$. Fix $\al\in\mc{E}_l$. Then by the above reduction $\al$ is not terminal, so $f_\al^l=\theta_{r_\al}\sum_{s\in \suc (\al)}f_s^l$ for some $\mc{S}_{r_\al}$-allowable $(f_s^l)$. By Fact \ref{f2} for $k=M-j-1$,  there exists $\mc{S}_{M-j-1-\ord(\al)}$-allowable functionals  $(h_{t})_{t\in A_{\alpha}}$ with 
$$
t(\alpha)f^l_{\alpha}(x)\leq \theta^{M-j-1}\sum_{t\in A_{\alpha}}h_{t}(x)\,.
$$
It follows that $(h_{t})_{t\in A_l}$ is $\mc{S}_{M-j-1}$-allowable, where $A_l=\cup_{\alpha\in E_{l}} A_{\alpha}$.
Now we have
\begin{align*}
\norm[E_lx]& = f_l( x)=\sum_{\al\in\mc{E}_l}t(\al)f_\al^l(E_lx)\\
&
\leq\sum_{\al\in \mc{E}_{l}}\theta^{M-j-1}
\sum_{t\in    A_{\alpha}}h_{t}(E_{l}x)= \theta^{M-j-1}\sum_{t\in A_l}h_t(E_lx).
\end{align*}
Taking into account the error from erasing nodes with too small orders we obtain
$$
\sum_l\norm[E_lx]\leq \theta^{M-j-1}\sum_l\sum_{t\in A_l}h_t(E_lx) +2\e\leq\dots
$$
Notice that $(h_t)_{t\in A}$ is $\mc{S}_{M-1}$-allowable. By Fact \ref{f1} with error $2\e/\theta_M$ we assume that the family $(h_t(x_i))_{t:h_t(x_i)\neq 0}$ is $\mc{S}_1$-allowable for each $i$ and thus we have:
\begin{align*}
\dots& \leq\theta^{M-j-1}\sum_l\sum_{i}a_i\sum_{\minsupp h_t\leq\minsupp x_i}h_t(E_lx_i)+4\e/\theta_M\\
& \leq\theta^{M-j-1}\theta^{-1}_1\sum_l\sum_{i}a_i\norm[E_lx_i]+4\e/\theta_M\\
&=\theta^{-1}_1\theta^{M-j-1}\sum_l\sum_{i}a_i\norm[E_lx_i]+4\e/\theta_M.
\end{align*}
\ep
In order to deal with allowable splittings, we need the next result, stating - roughly speaking - that a restriction of an average $x$ with an averaging tree high enough is still an average $y$, with a strict control on the error on the new average $y$ - depending on the error in the averaging tree of $x$ corresponding to $\minsupp y$.  
 \bl\label{theta1.5} Let $(x_i^j)$, $(N_i^j)$, $(a_i^j)$, $(\e_i^j)$ form an averaging tree for a $(M+\tilde{M},\e)$-average $x$, $M,\tilde{M}\in\N$, $\e>0$, of normalized block sequence $(x_i^0)_i$, satisfying 
\bnum
\item for any $i,j$ we have $N_i^j=2^{k_i^j}$ for some $k_i^j$,
\item for any $i,j$ we have $\e_{i+1}^j\leq\theta_M\e/2^i\maxsupp x_i^j$,
$\e_1^j\leq\theta_M\e/2$ for any $i,j$.
\enum
Then for any $I\subset\N$ with $N_{\min I}^M\sum_{i\in I}a_i^M\in\N$ the vector $y=\sum_{i\in I}a_i^Mx_i^M$ is a restriction of an $(M,\e_{\min I}^M)$-average of some block sequence $(y_k^0)$ with $\norm[y_k^0]\leq 1$ and such that the following property holds:

(P) for every $k,i,l$ either  $x_{i}^{M}\preceq y_{k}^{l}$ or $x_{i}^{M}\succeq y_{k}^{l}$ or $x_{i}^{M}$ and $y_{k}^{l}$ are incomparable, where $(y_k^l)_{k,l}$ is the family of nodes of averaging tree of $y$.
 \el
\bp Let $\e_I=\e_{\min I}^M$. We represent $y=\sum_{i\in I}a_i^Mx_i^M$ as a restriction of an $(M,\e_I)$-average.  We construct inductively on $l=M,M-1,\dots,0$ an averaging tree $(y_k^l)_{l=0,k=1}^{M,K_l}$ with weights $(W_k^l)$ and coefficients $(c_k^l)$, where $y_k^l=1/W_k^l\sum_{s\in J_k^l}y_s^{l-1}$ and $c_k^l=\prod_{r>l: y_k^l\preceq y_{k_r}^r}\frac{1}{W_{k_r}^r}$, such that $y^M_1=y$ and the following is satisfied
 \bnum
 \item [(P$_0$)] $c_k^ly_k^l=\sum_{m\in A_k^l}a_m^0x_m^0$, $c_k^l=\sum_{m\in A_k^l}a_m^0$ for every $k$  and $l<M$,
 \item [(P$_1$)] for every $k,i,l$ either $x_{i}^{l}\preceq y_{k}^l$ or $x_{i}^{l}$ is incomparable with $y_{k}^l$,
 \item[(P$_2$)] for every $i,j,k,l$ either  $x_{i}^{j}\preceq y_{k}^{l}$ or $x_{i}^{j}\succeq y_{k}^{l}$ or $x_{i}^{j}$ and $y_{k}^{l}$ are incomparable,
 \item[(P$_3$)] for every $k,l$ we have $W_k^l=\min\{N_i^l:\ x_i^l\preceq y_k^l\}$.
 \enum
We allow one difference from the original definition: $\# J_1^M=L=N_{\min I}^M\sum_{i\in I}a_i^M$, not $W_1^M$, to occur, otherwise $\# J_k^l=W_k^l$ for any $l<M$.

We let $y_1^M=\sum_{i\in I}a_i^Mx_i^M=\sum_{m\in A}a_m^0x_m^0$, $c_1^{M}=1$, $A_1^M=A$ and $W_1^M=N_{\min I}^M\leq\minsupp y$. All properties (P$_0$)-(P$_3$) are obviously satisfied.

Assume we have $(y_k^l)_k$, $(W_k^l)_k$ and $(c_k^l)_k$ for some $M\geq l>2$ satisfying the above.

Fix  $k$ and consider $A_k^l$. Pick any $m\in A_k^l$. By (P$_1$) in inductive assumption we have $x_{i_r}^r\preceq y_{k_r}^r$ for any $l\leq r\leq M$, $i_r,k_r$ with $x_m^0\preceq x_{i_r}^r$ and $x_m^0\preceq y_{k_r}^r$. Therefore $N_{i_r}^r\geq W_{k_r}^r$ for any $l\leq r\leq M$, $i_r,k_r$ as above. By Remark \ref{aat1} and (P$_{3}$) we have
$$
a_m^0 =\prod_{r=1}^M\frac{1}{N_{i_r}^r}\leq \prod_{r=l}^M\frac{1}{N_{i_r}^r} \leq \prod_{r=l}^M\frac{1}{W_{k_r}^r}=\frac{c_k^l}{W_k^l}\,.
$$
Recall that all coefficients $a_m^0, c_k^l, 1/W_k^l$ are some powers of $1/2$ and $(a_m^0)_m$ is non-increasing. Moreover for $l<M$ we have $\sum_{m\in A_k^l}a_m^0=c_k^l$, hence we can split $A_k^l$ into $W_k^l$-many successive sets $(A_s^{l-1})_{s=1}^{W_k^l}$ such that for each $s$ we have
$$
\sum_{m\in A_s^{l-1}}a_m^0=\frac{c_k^l}{W_k^l}\,.
$$
In case $l=M$ we have $\sum_{m\in A_1^M}a_m^0=L/W_1^M$, hence we can split $A_1^M$ into $L$-many sets $(A_s^{M-1})_{s=1}^{L}$ such that for each $s$ we have
$$
\sum_{m\in A_s^{M-1}}a_m^0=\frac{c_1^M}{W_1^M}=\frac{1}{W_1^M}\,.
$$
We define then $(y_s^{l-1})_s$  and $(c_s^{l-1})_s$ by
$$
\frac{c_k^l}{W_k^l}y_s^{l-1}=\sum_{m\in A_s^{l-1}}a_m^0x_m^0, \ \ \ \ c_s^{l-1}=\frac{c_k^l}{W_k^l}\,.
$$
Hence obviously $y_k^l=1/W_k^l\sum_sy_s^{l-1}$. We let also $W_s^{l-1}=\min\{N_i^{l-1}:\ x_i^{l-1}\preceq y_s^{l-1}\}$ and thus we finish construction of vectors on level $l-1$ satisfying (P$_0$) and (P$_3$). 

Now we verify property (P$_1$). Notice that by property (P$_{1}$) on level $l$ for each $k$ we have $\supp y_k^l=\cup\{\supp x_i^l: \ x_i^l\preceq y_k^l\}=\cup\{\supp x_s^{l-1}: \ x_s^{l-1}\preceq y_k^l\}$.  In case $l<M$ by Remark \ref{aat1} and (P$_0$) for $l$ we have
$$
\sum_{r:\ y_r^{l-1}\preceq y_k^l}c_r^{l-1}=W_k^l\frac{c_k^l}{W_k^l}=c_k^l=\sum_{m\in A_k^l}a_m^0=\sum_{s:\ x_s^{l-1}\preceq y_k^l}a_s^{l-1}\,,
$$
and as in the construction each $a_s^{l-1}\leq c_k^l/W_k^l=c_r^{l-1}$. In case of $l=M$ we have
$$
\sum_{r:\ y_r^{M-1}\preceq y_k^M}c_r^{M-1}=L\frac{c_1^M}{W_1^M}=\frac{L}{W_k^l}=\sum_{m\in A_1^M}a_m^0=\sum_{s:\ x_s^{M-1}\preceq y_k^M}a_s^{M-1}\,,
$$
and each $a_s^{M-1}\leq 1/W_1^M=c_r^{M-1}$. Since all coefficients are the powers of $1/2$ and the sequence $(a_s^{l-1})_s$ is non-increasing we can partition the set $\{s: \ x_s^{l-1}\preceq y_k^l\}$ into $\cup\{B_r: \ y_r^{l-1}\preceq y_k^l\}$ such that for any $r$ we have $c_r^{l-1}=\sum_{s\in B_r}a_s^{l-1}$. Consequently for any $y_r^{l-1}\preceq y_k^l$ and $x_s^{l-1}\preceq y_k^l$ we have either $y_r^{l-1}\succeq x_s^{l-1}$ or $y_r^{l-1}$ and $x_s^{l-1}$ are incomparable.

The property (P$_2$) can be verified analogously by induction. If for some $l,k,j$ we have $\supp y_k^l=\cup\{\supp x_i^j: \ x_i^j\preceq y_k^l\}$, then we show that for any $y_r^{l-1}\preceq y_k^l$ and $x_s^{j-1}\preceq y_k^l$ we have either $y_r^{l-1}\succeq x_s^{j-1}$ or $y_r^{l-1}$ and $x_s^{j-1}$ are incomparable. The same argument works if $\supp x_i^j=\cup\{\supp y_k^l: \ x_i^j\succeq y_k^l\}$ for some $i,j,l$.

Define for each $l=M,\dots,1$ and $k=1,\dots,K_l$ the error $\delta_k^l$. For $k=1$ let $\delta_1^l=\e_I$, for any $l=M,\dots,1$. By property (P$_1$) for any $l,k$ there is some $i_k\geq k$ with
$$
\maxsupp y_{k}^l\leq\maxsupp x_{i_k}^l <\minsupp x_{i_k+1}^l\leq \minsupp y_{k+1}^l\,.
$$
Let $\delta_{k+1}^l=\e_{i_k+1}^l$ for any $k\geq 1$. We verify condition $(5)$ of Definition \ref{aat0}. For $k=1$ and $l=M,\dots,1$ we have $W_1^l\geq N_{\min I}^M\geq 2/\e_{\min I}^M=2/\delta_1^l$. On the other hand we have for any $l=M-1,\dots,1$ and $k=1,\dots,K_l-1$
$$
\delta_{k+1}^l=\e_{i_k+1}^l<1/2^{i_k}\maxsupp x_{i_{k}}^l\leq 1/2^k\maxsupp y_k^l\,,
$$
and $W_{k+1}^l\geq N_{i_k+1}^l>2/\e_{i_k+1}^l=2/\delta_{k+1}^l.$ 

Hence $(y_k^l)_{k,l}$, $(W_k^l)_{k,l}$, $(c_k^l)_{k,l}$, $(\delta_k^l)_{k,l}$ form an averaging tree and thus $y$ is $(M,\e_I)$-average of $(y_k^0)_k$.
Notice that
$$
\norm[c_k^0y_k^0]=\norm[\sum_{m\in A_k^0}a_m^0x_m^0]\leq \sum_{m\in A_k^0}a_m^0=c_k^0\,,
$$
therefore $\norm[y_k^0]\leq 1$. Moreover property (P$_{2}$) includes property (P). \end{proof}
\begin{remark}\label{ns2}
Note that by the construction each sequence $(y_s^{l-1})_{s\in J_k^l}$ is $\mc{S}_1$-admissible for any $k,l$. Hence it readily follows that for every set $F$ of incomparable nodes $(y_k^l)$ the functional $\sum_{y_k^l\in F}\theta^{M-l}e^{*}_{\minsupp y_k^l}$ is a norming functional on the space $T[\mc{S}_{1},\theta]$.
\end{remark}
The next Lemma provides a "Tsirelson-type" upper estimate for the norms of averages.
 \bl\label{theta1.7} Let $(x_i^j)$, $(N_i^j)$, $(a_i^j)$, $(\e_i^j)$ form an averaging tree for a $(2M-3,\e)$-average $x$, $M>1$, $\e>0$, of normalized block sequence $(x_i^0)_i$, satisfying additionally the following conditions:
\bnum
\item for any $i,j$ we have $N_i^j=2^{k_i^j}$ for some $k_i^j$,
\item for any $i,j$ we have $\e_{i+1}^j\leq\theta_M\e/2^i\maxsupp x_i^j$,
$\e_1^j\leq\theta_M\e/2$ for any $i,j$.
\enum
Fix an $\mc{S}_{M-4}$-allowable family $\mc{E}$ of subsets of $\N$, such that the family $\{E\in\mc{E}: \ Ex_i^M\neq 0\}$ is $\mc{S}_1$-allowable for any $i$,  and coefficients $(t_E)_{E\in\mc{E}}\subset [0,1]$.

Then there is a partition $(V_E)_{E\in\mc{E}}$ of nodes $(x_i^0)_i$, with $\minsupp x_{\min V_E}^0\geq \min E$, such that
$$
\sum_{E\in\mc{E}}t_E\norm[Ex]\leq C\sum_{E\in\mc{E}}t_E \norm[\sum_{i\in V_E}a_i^0e_{\minsupp x_i^0}]_{T[\mc{S}_1,\theta]}+C\e
$$
for some universal constant $C$ depending only on $\theta_1$ and $\theta$.
 \el
\bp
\textsf{STEP 1}. Let us recall that  $x$ is an $(M-3,\e)$-average of $(x_i^M)_i$.
First let $\mc{E}_i=\{E\in\mc{E}: \ E \text{ begins at } x_i^M\}$ and $J=\{i:\ \mc{E}_i\neq\emptyset\}$. As $(x_i^M)_{i\in J}$ is $\mc{S}_{M-4}$-admissible, we have
\begin{align*}
\sum_{E\in\mc{E}}t_E\norm[E\sum_{i\in J}a_i^Mx_i^M]\leq \sum_{i\in J}a_i^M\sum_{E\in\mc{E}}\norm[Ex_i^M]\leq\theta_1^{-1}\sum_{i\in J}a_i^M\leq\theta_1^{-1}2\e.
\end{align*}
For any $E\in\mc{E}$ let $I_E=\{i\not\in J: \ Ex_i^M\neq 0\}$, $i_E=\min I_E$ and $\e_E=\e^M_{i_E}$. Compute
\begin{align*}
\sum_{E\in\mc{E}}\e_E&\leq \e\theta_M\sum_{i\in J}\sum_{E\in\mc{E}_i}1/2^{i_E-1}\maxsupp x_{i_E-1}^M\\
&\leq \e\theta_M\sum_{i\in J}\maxsupp x_i^M/2^{i}\maxsupp x_{i}^M\leq \e\theta_M.
\end{align*}
\textsf{STEP 2}. Fix $E\in\mc{E}$.
Let $\sum_{i\in I_E}a_i^Mx_i^M=\sum_{m\in K}a_m^0x_m^0$. Notice that each $a_m^0\leq 1/N_{i_E}^M$ and $(a_m^0)_m$ is non-increasing, therefore we can partition $K$ into intervals $A<B$ with $\sum_{m\in A}a_m^0=L/N_{i_E}^M$ and $\sum_{m\in B}a_m^0=\delta/N_{i_E}^M$ for some $L\in\N$ and $0\leq \delta<1$. Hence we can erase $\sum_{m\in B}a_m^0x_m^0$ with error $\delta/N_{i_E}^M\leq 1/N_{i_E}^M\leq \e_E$.

After this reduction by Lemma \ref{theta1.5} the vector $y=\sum_{i\in I_E}a_i^Mx_i^M$ is a restriction of an $(M-2,\e_E)$-average $\sum_kc_k^2y_k^2$ with $\norm[y_k^2]\leq 1$ and property (P) given by a suitable averaging tree $(y_k^l)_{k,l}$ with proper weights, coefficients and errors. 

We take the family $K=\{k:$  $\minsupp x_{i}^{M}\in\ran y_{k}^{2}$ for some $x_{i}^{M}\}$. 
Since $(x_i^M)_i$ is an $\mc{S}_{M-3}$-admissible family and $y$ is an $(M-2,\e_E)$-average of $(y_k^2)$, we can erase $\sum_{k\in K}c_k^2y_k^2$ with error $2\e_E$.
For any $i$ let 
$$
l_{E,i}=\min\{M\geq l\geq 0:  y_k^l\succeq  x_i^M \}.
$$
By the above reduction and (P) we can assume that $l_{E,i}\geq 2$ for all $i\in I_E$. Let 
\begin{center}
\mbox{ $K_{E,i}=\{k:  y_k^2\preceq x_i^M\}$ for any $i\in I_E$.}
\end{center}
Compute by Lemma \ref{theta1} for the $(M-2,\e_E)$-average $\sum_{k}c_{k}^2y_{k}^{2}$ and $j=0$
\begin{align*}
\norm[Ex]=\norm[E\sum_{k}c_{k}^2y_{k}^{2}] & \leq\norm[\sum_{k\not\in K}c_k^2Ey_k^2]+2\e_E\\
&\leq \theta_1^{-1}\theta^{M-3}\sum_{i\in I_E}\sum_{k\in K_{E,i}}c_k^2\norm[Ey_k^2]+6\e_E/\theta_M\\
& =\theta_1^{-1}\theta^{M-3}\sum_{i\in I_E}a_i^M\sum_{k\in K_{E,i}}\norm[\frac{c_k^2}{a_i^M}Ey_k^2]+6\e_E/\theta_M
.
\end{align*}
\textsf{STEP 3}. Fix $i\not\in J$. Put $\mc{F}_i=\{E\in\mc{E}: \ i\in I_E\}=\{E\in\mc{E}: Ex_i^M\neq 0\}$. For any $E\in\mc{F}_i$ and $k\in K_{E,i}$ let $w_k=\frac{c_k^2}{a_i^M}Ey_k^2$. For each $k\in K_{E,i}$ take the norming functional $f_k$ with $f_k(w_k)=\norm[w_k]$ and $\supp f_{k}\subset\supp w_{k}.$

We gather all the terminal nodes in the tree-analysis of $f_k$ for all $k\in K_{E,i}$, $E\in\mc{F}_i$, of order smaller than $M-{l_{E,i}}$. By the assumption on $\mc{E}$ and the fact that $l_{E,i}\geq 2$ they form an $\mc{S}_{M-1}$-allowable family, hence as $x_i^M$ is an $(M,\e_i^M)$-average, we can erase these nodes with total error $2\e_i^M$.

By Fact \ref{f2}, adding nodes in the tree-analysis of each $f_k$, $k\in K_{E,i}$, on the level $M-l_{E,i}$, we get $\norm[w_k]\leq \theta^{M-l_{E,i}}\sum_lf_k^l(x_i^M)$ for some $\mc{S}_{M-l_{E,i}}$-allowable functionals $(f_{k}^{l})_{l}$. Pick $E_i$ with $t_{E_i}\theta^{-l_{E_i,i}}=\max\{t_{E}\theta^{-l_{E,i}}: E\in\mc{E}\}$. Let $l_i=l_{E_i,i}$ and compute
\begin{align*}
\sum_{E\in\mc{F}_i}t_E \sum_{k\in K_{E,i}}\norm[\frac{c_k^2}{a_i^M}Ey_k^2] &\leq\sum_{E\in\mc{F}_i}t_E\theta^{M-l_{E,i}}\sum_{k\in K_{E,i}}\sum_{l}f_k^l(x_i^M)+2\e_i^M\leq \dots
\end{align*}
Notice again that $(f_k^l)_{l,k\in K_{E,i}, E\in\mc{E}}$ is an $\mc{S}_{M-1}$-allowable family (as before by $l_{E,i}\geq 2$ and assumption on $\mc{E}$). As $x_i^M$ is an $(M,\e_i^M)$-average of suitable $(x_m^0)_m$, by Fact \ref{f1} with error $2\e_i^M/\theta_M$, we may assume that for any $m$ the family $(\supp f_k^l\cap\supp x_m^0)_{l,k\in K_{E,i}, E\in \mc{E}}$ is $\mc{S}_1$-allowable.  Therefore we continue the estimation
\begin{align*}
\dots &\leq t_{E_i}\theta^{M-l_i}\sum_{E\in\mc{F}_i}\sum_{k\in K_{E,i}}\sum_lf_k^l(x_i^M)+4\e_i^M/\theta_M\leq \theta_1^{-1}t_{E_i}\theta^{M-l_i}+4\e_i^M/\theta_M\,.
\end{align*}
\textsf{STEP 4}. We define $J_E=\{i:\ E=E_i\}\subset I_E$ for any $E\in\mc{E}$. Notice that $(J_E)_{E\in\mc{E}}$ are pairwise disjoint.  By STEP 1, STEP 2 and STEP 3 we have
\begin{align*}
\sum_{E\in\mc{E}}t_E\norm[Ex] & \leq \sum_{E\in\mc{E}}t_E\norm[E \sum_{i\in J}a_i^Mx_i^M]+\sum_{E\in\mc{E}}t_E\norm[E\sum_{i\in I_E}a_i^Mx_i^M]\\
& \leq 2\theta_1^{-1}\e+\theta_1^{-1}\theta^{M-3}\sum_{E\in\mc{E}}\sum_{i\in I_E}a_i^M\sum_{k\in K_{E,i}}t_E\norm[\frac{c_k^2}{a_i^M}Ey_k^2]+6\sum_{E\in\mc{E}}\e_E/\theta_M\\
& = \theta_1^{-1}\theta^{M-3}\sum_{i\not\in J}a_i^M\sum_{E\in\mc{F}_i}\sum_{k\in K_{E,i}}t_E\norm[\frac{c_k^2}{a_i^M}Ey_k^2]+(6+2\theta_1^{-1})\e \\
 &\leq \theta_1^{-2}\theta^{M-3}\sum_{i\not\in J}a_i^Mt_{E_i}\theta^{M-l_i}+4\sum_{i}\e_i^M/\theta_M+(6+2\theta_1^{-1})\e\\
&\leq \theta^{-2}_1\theta^{M-3}\sum_{E\in\mc{E}}t_E\sum_{i\in J_E}a_i^M\theta^{M-l_i}+(10+2\theta_1^{-1})\e \leq\dots
\end{align*}
Fix $E\in\mc{E}$.  Notice that for any $l$ the sequence $(x_i^M)_{x_i^M\preceq y_k^l,l=l_{i}}$ is $\mc{S}_1$-admissible, hence by  Remark~\ref{ns2} the formula $\sum_{i\in I}\theta^{M-l_{i}+1}e^*_{\minsupp x_i^M}$ defines a norming functional in $T[\mc{S}_1,\theta]$.  Therefore for any $E\in\mc{E}$ we have
$$
\sum_{i\in J_E}a_i^M\theta^{M-l_i}\leq\theta^{-1}\norm[\sum_{i\in J_E}a_i^Me_{\minsupp x_i^M}]_{T[\mc{S}_1,\theta]}
\,,
$$ 
and  we continue the above estimation
\begin{align*}
\dots& \leq \theta^{-2}_1\theta^{M-4}\sum_{E\in\mc{E}}t_E\norm[\sum_{i\in J_E}a_i^Me_{\minsupp x_i^M}]_{T[\mc{S}_1,\theta]}+(10+2\theta_1^{-1})\e\leq\dots
\end{align*}
Consider $z_{i}^M=1/a_i^M\sum_{x_m^0\preceq x_i^M}a_m^0e_{\minsupp x_m^0}$, for $i=1,\dots,N^M$, which are $(M,\e_i^M)$-averages in $T[\mc{S}_1,\theta]$ by Remark \ref{aat1}. 
As $\norm[z_{i}^M]_{T[\mc{S}_1,\theta]}\geq \theta^M$ for each $i$, we continue
\begin{align*}
\dots & \leq\theta^{-2}_1\theta^{-4}\sum_{E\in\mc{E}}t_E\norm[\sum_{i\in J_E}a_i^Mz_{i}^M]_{T[\mc{S}_1,\theta]}+(10+2\theta^{-1}_1)\e\\
& \leq\theta^{-2}_1\theta^{-4}\sum_{E\in\mc{E}}t_E\norm[\sum_{i\in J_E}\sum_{x_m^0\preceq x_i^M}a_m^0e_{\minsupp x_m^0}]_{T[\mc{S}_1,\theta]}+C\e\,,
\end{align*}
 which ends the proof with $C=10+2\theta^{-2}_1\theta^{-4}$ and $V_E=\{m: x_m^0\preceq x_i^M, i\in J_E\}$ for each $E\in\mc{E}$. 
\ep
\subsection{Special types of averages}
We present the lower "Tsirelson-type" estimate in a regular modified mixed Tsirelson space $X$ with $(\clubsuit)$. In order to achieve this we need special types of averages. We start with Corollary 4.10  \cite{mp} recalled below

\bpr\label{theta2.5} For any block subspace $Y$ of $X$, any $M\in\N$ and $\e>0$, there is an $(M,\e)$-average $x\in Y$ of some normalized block sequence in $Y$ such that 
$$
\theta^{M-j}D\geq\sup\left\{\sum_{i}\nrm{E_ix}:\ \mc{S}_j\text{-allowable }(E_i)\right\}\geq \theta^{M-j}/D
$$
for any $0\leq j\leq M$ and some universal constant $D$ depending only on $\theta_1$ and $\theta$. 
\epr
\bp We recall Lemma 4.9 \cite{mp}, whose proof is valid, line after line, also in the modified case. Lemma 4.9 \cite{mp} and Lemma \ref{theta1} yield the Proposition. 
\ep
\bd A special $(M,\e)$-average $x$,  $M\in\N$, $\e>0$, is any $(M,\e)$-average satisfying assertion of Proposition \ref{theta2.5}.
\ed
For the next lemma we shall need the following observation.
\bfa\label{height} Fix $M\in\N$. Then for any $G\in\mc{S}_M$ and any $z=\sum_{i\in G}a_ie_i\in T[\mc{S}_1,\theta]$, $(a_i)_{i\in G}\subset [0,1]$, there is a norming functional $f$ with a tree-analysis with height at most $M$, such that $\norm[z]_{T[\mc{S}_1,\theta]}\leq 2f(z)$.
\efa 
\bp
Take a norming functional $g$ with a tree-analysis $(g_t)_{t\in T}$ satisfying $g(z)=\norm[z]_{T[\mc{S}_1,\theta]}$. Let $I$ be the set of all terminal nodes of $\mt$ with order at most $M$ and let  $g_1$ be the restriction of $g$ to $I$ and $g_2=g-g_1$. If $g_1(z)\geq g_2(z)$ then we let $f=g_1$. Assume that $g_1(z)\leq g_2(z)$ and compute
$$
 g(z)\leq 2g_2(z)\leq 2\theta^{M+1}\sum_{i\in G\setminus I}a_i\leq 2\theta^M\sum_{i\in G}a_i=2f(z)\,,
$$
where $f=\theta^M\sum_{i\in G}e^*_i$, which ends the proof.
\ep
The major obstacle in obtaining the lower "Tsirelson-type" estimate for norm is the fact that given an $(M,\e)$-average $x=\sum_{i\in F}a_ix_i$ we do not control the norm of $\sum_{i\in G}a_ix_i$, $G\subset F$, in general case. The next result provides a block sequence $(x_i)$ whose any $\mc{S}_M$-admissible subsequence dominates suitable subsequence of the basis in the original Tsirelson space.  
\begin{lemma}\label{theta3}
For every block subspace $Y$ and every  $M\in\N$, $\delta>0$, there exists a block sequence  $(x_i)$ of $Y$ satisfying for any $G\in \mc{S}_{M}$ and scalars $(a_{i})_{i\in G}$
\begin{equation}\label{thetae}
\norm[\sum_{i\in G}a_{i}x_{i}] \geq \frac{1}{2}(1-\delta) \norm[\sum_{i\in G}a_{i}\norm[x_{i}]e_{\minsupp x_i}]_{T[\mc{S}_{1},\theta]}.
\end{equation}
\end{lemma}
\begin{proof}
Assume the contrary. Notice first that for any $M\in\N$ we have
$$
(\sqrt[m]{\theta_m})^M\leq \sqrt[m]{\theta_{Mm}}\leq\sqrt[m]{\theta^{mM}}\,,
$$
thus $\lim_{m\to\infty}\sqrt[m]{\theta_{Mm}}=\theta^M$. Pick $m\in\N$ such that $\sqrt[m]{\theta_{Mm}}>\sqrt[m]{D^2}(1-\delta)\theta^{M}$ with $D$  as in Prop. \ref{theta2.5}. Take a block sequence $(x_{i}^0)_{i}$ of special $(Mm,\e)$-averages, for some $\e>0$.

Since \eqref{thetae} fails there is an infinite sequence $G_{k_1}^{1}$ of successive elements of $\mc{S}_{M}$ and coefficients $(a_{i}^{1})_{i\in G_{k_1}^{1}}$ such that 
$$
\norm[\sum_{i\in G_{k_1}^{1}}a_{i}^{1}x_{i}^0]<\frac{1}{2}(1-\delta) \norm[\sum_{i\in G_{k_1}^1}a^{1}_{i}\norm[x_{i}^0]e_{m^0_{i}}]_{T[\mc{S}_{1},\theta]},
$$
where $m_{i}^0=\minsupp x_i^0$ for each $i$. Set $x_{k_1}^{1}=\sum_{i\in G_{k_1}^{1}}a_{i}^{1}x^0_{i}$, $k_1\in\N$, and by Fact \ref{height} take  norming functionals $f^1_{k_1}$ of the space $T[\mc{S}_{1},\theta]$ of height at most $M$ with
$$
\norm[\sum_{i\in G_{k_1}^1 } a^{1}_{i}\norm[x_{i}^0]e_{m_{i}^0}]_{T[\mc{S}_{1},\theta]} \leq 2f_{k_1}^{1}\left(\sum_{i\in G_{k_1}^1}a^{1}_{i}\norm[x_{i}^0]e_{m_{i}^0}\right).
$$
Assume that we have defined  $(x_{k_{j-1}}^{j-1})_{k_{j-1}}$ and $(f_{k_{j-1}}^{j-1})_{k_{j-1}}$ for some $j<m$. Then the failure of \eqref{thetae}  implies the existence of a sequence $(G_{k_j}^{j})_{k}$ of successive elements of $\mc{S}_{M}$ and a sequence $(a_{i}^{j})_{\in G_{k_{j}}^{j}}$ such that
$$
\norm[\sum_{i\in G_{k_{j}}^{j}}a_{i}^{j}x^{j-1}_{i}]<\frac{1}{2}(1-\delta) \norm[\sum_{i\in G_{k_{j}}^{j}}a^{j}_{i} \norm[x^{j-1}_{i}]e_{m_{i}^{j-1}}]_{T[\mc{S}_{1},\theta]}\,,
$$
where $m_{i}^{j-1}=\minsupp x_i^{j-1}$. Set $x_{k_j}^j=\sum_{i\in G_{k_{j}}^{j}}a_{i}^{j}x^{j-1}_{i}$, for $k_j\in\N$, and take norming trees $f_{k_{j}}^{j}$ of the space $T[\mc{S}_{1},\theta]$ of height at most $M$ such that
$$
\norm[\sum_{i\in G_{k_{j}}^{j}}a^{j}_{i}\norm[x_{i}^{j-1}]e_{m_{i}^{j-1}}]_{T[\mc{S}_{1},\theta]}\leq 2f_{k_{j}}^{j}\left(\sum_{i\in G_{k_{j}}^{j}} a^{j}_{i}\norm[x_{i}^{j-1}]e_{m_{i}^{j-1}}\right).
$$
The inductive construction ends once we get the vector $x_{1}^{m}$ and the functional $f_1^m$.

Each functional $f_{k_j}^j$ is of the form $\sum_{i\in G_{k_{j}}^{j}} \theta^{l^{j}_{i}}e^*_{m_{i}^{j-1}}$, by construction satisfying
$$
\norm[x_{k_j}^j]<(1-\delta)\sum_{i\in G_{k_{j}}^{j}} \theta^{l^{j}_{i}}a^{j}_{i}\norm[x_{i}^{j-1}]\,.
$$
Inductively, beginning from $f_{1}^{m}$ we produce a tree-analysis of some norming functional $f$ on $T[\mc{S}_{1},\theta]$ by substituting each terminal node $e^*_{m_{k}^{j}}$, $j=1,\dots,m$, by the tree-analysis of the functional $f_{k}^{j}$.

Put $G=\cup_{k_{m-1}\in G_{1}^m}\cup_{k_{m-2}\in G_{k_{m-1}^{m-1}}}\dots\cup_{k_1\in G_{k_2}^2}G_{k_1}^1$. Let $(l_i)_{i\in G}$ be such that $f=\sum_{i\in G}\theta^{l_i}e^*_{m_{i}^0}$. Notice that $l_i\leq mM$ for any $i\in G$, as the height of each $f_i^j$ does not exceed $M$. 
We compute the norm of $x_1^m$, which is of the form
$$
x_1^m=\sum_{k_{m-1}\in G_{1}^m}\sum_{k_{m-2}\in G_{k_{m-1}^{m-1}}}\dots\sum_{k_1\in G_{k_2}^2}\sum_{i\in G_{k_1}^1}a_{k_{m-1}}^{m}\dots a_i^1x_i^0=\sum_{i\in G}b_ix_i^0\,.
$$
 Since each $x_i^0$ is a special $(mM,\e)$-average,  for some $\mc{S}_{mM-l_i}$-allowable sequence $(E_l)_{l\in L_i}$ we have $\norm[x_i^0]\leq D^2\theta^{mM-l_i}\sum_{l\in L_i}\norm[E_lx_i^0]$.

We have on one hand by the above construction
\begin{align*}
\norm[x_1^m]&\leq (1-\delta)^m\sum_{i\in G}\theta^{l_i}b_i\norm[x_i^0] \\
& \leq (1-\delta)^mD^2\sum_{i\in G}\theta^{l_i}b_i\theta^{mM-l_i}\sum_{l\in L_i}\norm[E_lx_i^0]\\
&=(1-\delta)^mD^2\theta^{mM}\sum_{i\in G}b_i\sum_{l\in L_i}\norm[E_lx_i^0]\,.
\end{align*}
Notice that $(E_l)_{l\in\cup_{i\in G}L_i}$ is $\mc{S}_{mM}$-allowable by the definition of $f$ and $(l_i)_{i\in G}$, thus 
$$
\norm[x_1^m]\geq \theta_{mM}\sum_{i\in G}b_i\sum_{l\in L_i}\norm[E_lx_i^0]\,,
$$
which brings $\theta_{mM}\leq (1-\delta)^mD^2\theta^{mM}$, a contradiction with the choice of $m$.
\end{proof}
\bd A Tsirelson $(M,\e)$-average $x$,  $M\in\N$, $\e>0$, is an $(M,\e)$-average $x=\sum_{i\in F}a_ix_i$ of a normalized block sequence $(x_i)$ satisfying the assertion of the Lemma \ref{theta3} with $\delta=1/2$. 
\ed
\bd A RIS of (special, Tsirelson) averages is any block sequence of (special, Tsirelson) $(n_k,\e/2^k)$-averages $(x_k)$ for $\e>0$ and $(n_k)_k\subset \N$ satisfying
$$
\theta_{l_{k+1}}\norm[x_k]_{\ell_1}\leq \frac{\e}{2^{k+1}}, \ \ \ k\in\N\,,
$$
where $l_{k}=\max\{l\in\N: \ 4l\leq n_k\}$, $k\in\N$.
\ed
We need the following technical lemma, mostly reformulating Lemma 7, \cite{klmt}:
\begin{fact}\label{ave0} Take RIS of normalized averages $(x_k)$, for some $(n_k)\subset\N$ and $\e>0$, and some $x=\sum_kb_kx_k$ with $(b_k)\subset [0,1]$. Then for any norming functional $f$ with a tree-analysis $(f_\al)_{\al\in\mt}$ there is a subtree $\mt'$ such that the corresponding functional $f'$ defined by the tree-analysis $(f_\al)_{\al\in\mt'}$ satisfies $f(x)\leq f' (x)+3\e$ and the following holds for any $k$
\begin{enumerate}
 \item[(a)] any node $\al$ of $\mt'$ with $f_\al(x_k)\neq 0$ satisfies $\ord (\al)< n_{k+1}/4$, 
\item[(b)] any terminal node $\al$ of $\mt'$ with $f_\al(x_k)\neq 0$ satisfies $\ord(\al)\geq n_k$.
\end{enumerate}
\end{fact}
\bp In order to prove (a) we repeat the reasoning from the proof of Lemma 7 \cite{klmt}. For any $k$ let $\mc{F}_k$ be the collection of all nodes in $\mt$ which are minimal with respect to the property $\ord (\al)\geq n_{k+1}/4$ and $f_\al(x_k)\neq 0$. Then
$$
\sum_{\al\in\mc{F}_k}t(\al)f_\al(x_k)\leq \theta_{l_{k+1}}\norm[x_k]_{\ell_1}\leq \frac{\e}{2^{k+1}}\,.
$$
Thus we can erase all nodes from $\mc{F}_k$ restricted to supports of $x_k$, for all $k$, with error $\sum_kb_k\frac{\e}{2^{k+1}}\leq \e$.  

For (b) we use Fact \ref{f3} for erasing all terminal nodes $\al$ of $\mt$ with $f_\al(x_k)\neq 0$ with error $2\e_k$, for any $k$. 
\ep
\bl\label{ave1} Let $x=\sum_ka_kx_k$ be an $(M,\e)$-average of RIS of normalized special averages $(x_k)$, for $(n_k)\subset\{ M+3,M+4,\dots\}$ and $\e>0$, with $\e<\theta_M$.

Then $\norm[x]\leq D'\theta_M$, for some universal constant $D'$ depending only on $\theta$ and $\theta_1$. 
\el
\bp Take a norming functional $f$ with a tree-analysis $(f_\al)_{\al\in\mt}$ such that $\norm[x]=f(x)$. Using Fact  \ref{ave0} pick the subtree $\mt'$ satisfying (a) and (b) and the corresponding functional $f'$. 

Let $\mc{E}$ be collection of all $\al\in\mt'$ maximal with respect to the property $\ord (\al)\leq M-1$. Notice that $\mc{E}$ is $\mc{S}_{M-1}$ - allowable. 

Fix $\al\in\mc{E}$. Then $\al$ is not terminal, so $f_\al=\theta_{r_\al}\sum_{s\in \suc (\al)}f_s$. As in Fact \ref{f2} we  partition $\suc (\al)=\bigcup_{t\in A_\al}F_t$ in such a way that $(f_s)_{s\in F_t}$ is $\mc{S}_{\ord(s)-(M-1)}$-allowable for every $t\in A_\al$ and $(g_t)_{t\in A_\al}$ is $\mc{S}_{M-1-\ord(\al)}$-allowable, where $g_t=\sum_{s\in F_t}f_s$. Let $A=\cup_{\al\in\mc{E}}A_\al$ and notice that $(g_t)_{t\in A}$ is $\mc{S}_{M-1}$-allowable. Let $H$ denote the set of all $k$ such that some $g_t$, $t\in A$, begins in $x_k$. Since $x$ is an $(M,\e)$-average we have $\norm[\sum_{k\in H}a_kx_k]\leq\sum_{k\in H}a_k\leq 2\e$.

By definition of $H$ for any $\al\in\mc{E}$ and $k\not\in H$ with $f_\al(x_k)\neq 0$ there is an immediate successor of $\al$ beginning before $x_k$. Thus by (a)  we have for any $k\not\in H$
\begin{enumerate}
\item[(c)] for any $\al\in\mc{E}$ with $f_\al(x_k)\neq 0$ the order of immediate successors of $\al$ is at most $n_k/4$,
 \item[(d)]  $\{ g_t: \ t\in A,\ g_t(x_k)\neq 0\}$ restricted to $\supp x_k$ is $\mc{S}_1$ - allowable.
\end{enumerate}
Fix $k\not\in H$ and $t\in A$ with $g_t(x_k)\neq0$ and let $B_t^k=\{s\in F_t:\ f_s(x_k)\neq 0\}$.

Fix $s\in B_t^k$ and take the subtree $\mt_s$ of $\mt'$ consisting of $s$ (as a root) and of all successors of $s$ in $\mt'$. By Fact~\ref{f2}, using (b) and (c) we can add nodes in $\mt_s$ on level $n_k-\ord(s)$ obtaining  $(h_{s,r})_{r\in C_s}$ which is $\mc{S}_{n_k-\ord(s)}$-allowable satisfying 
$$
f_s(x_k)\leq \sum_{r\in C_s}\theta^{n_k-\ord(s)} h_{s,r}(x_k).
$$
Compute for $k\not\in  H$ using the above and $(\clubsuit)$
\begin{align*}
f'(x_k) & =\sum_{t\in A}\sum_{s\in B_t^k}t(s)f_s(x_k)\\
& \leq \theta_M\sum_{t\in A}\sum_{s\in B_t^k}\theta^{\ord(s)-M}\sum_{r\in C_s}\theta^{n_k-\ord(s)}h_{s,r}(x_k) \\
& \leq \theta_M\sum_{t\in A}\sum_{s\in B_t^k}\sum_{r\in C_s}\theta^{n_k-M}h_{s,r}(x_k)\,.
\end{align*}
Notice that the family $\{h_{s,r}:\ r\in C_s,\ s\in B_t^k\}$ for any fixed $t\in A, k\not\in  H$ is $\mc{S}_{n_k-M+1}$-allowable. Therefore by (d) the family $\{h_{s,r}:\ r\in C_s,\ s\in B_t^k, \ t\in A\}$ for any fixed $k\not\in  H$ is $\mc{S}_{n_k-M+2}$-allowable and hence since $x_k$ is a normalization of a $(n_k,\e_k)$-special average, we continue the estimation
$$
\dots\leq \theta_M\theta^{n_k-M}D^2\theta^{-n_k+M-2}=D^2\theta^{-2}\theta_M\,.
$$
We compute
$$
f(x)\leq f'(x)+3\e\leq\sum_{k\not\in H}a_k f( x_k)+5\e\leq D^2\theta^{-2}\theta_M+5\e\leq (D^2\theta^{-2}+5)\theta_M\,,
$$
which ends the proof of Lemma. 
\ep
\subsection{Main results}
\bt \label{dist} Let $X$ be a regular modified mixed Tsirelson space $T_M[(\mc{S}_n,\theta_n)_n]$. If $\theta_n/\theta^n\searrow 0$, then $X$ is arbitrary distortable. 
\et
\bp Theorem follows immediately from Proposition \ref{theta2.5} and Lemma \ref{ave1}. 
\ep
Recall that a Banach space $X$ with a basis is called sequentially minimal (\cite{fr}), if any block subspace of $X$ contains a block sequence $(x_n)$ such that every block subspace of $X$ contains a copy of a subsequence of $(x_n)$. Notice that this property implies quasiminimality of $X$.
\bt\label{quasi} Let $X$ be a  regular modified mixed Tsirelson space $T_M[(\mc{S}_n,\theta_n)_n]$. If $\theta_n/\theta^n\searrow$, then $X$ is sequentially minimal. 
\et
The theorem follows immediately from  the following result:
\bl\label{ris} Let $(x_k)_k$, $(y_k)_k$ be RIS of Tsirelson $(2M_k-3,\e_k)$-averages, $M_k>4$, $\e<(6C)^{-1}$, with $C$ as in Lemma \ref{theta1.7}, such that
 \bnum
 \item $x_k$ has an averaging tree $(x_{k,i}^j)_{i,j}$, $(N_{k,i}^j)_{i,j}$, $(\e_{k,i}^j)_{i,j}$,  $(a_{k,i}^j)_{i,j}$, $y_k$ has an averaging tree $(y_{k,i}^j)_{i,j}$, $(N_{k,i}^j)_{i,j}$, $(\e_{k,i}^j)_{i,j}$,  $(a_{k,i}^j)_{i,j}$, both satisfying conditions (1) and (2) of Lemma \ref{theta1.7} for any $k$, 
 \item $\minsupp x_{k,i}^0=\minsupp y_{k,i}^0$ and $\norm[x_{k,i}^0]=\norm[y_{k,i}^0]=1$ for any $k,i$,
 \item $\e_k\leq\theta_{2M_k-3}\theta^{2M_k-3}\e/2^{k+2}$ for any $k$.
 \enum
Then $(x_k/\norm[x_k])_k$ and $(y_k/\norm[y_k])_k$ are equivalent.
 \el
Notice first that Lemma above yields Theorem \ref{quasi}, as given a block sequence $(w_n)$ in $X$ and a block subspace $Y$ of $[w_n]$ and $k\in\N$, we can choose block sequences $(u_i)\subset [w_n]$ and $(v_i)\subset Y$ satisfying the assertion of Lemma \ref{theta3} for $2M_k-3$. Passing to a subsequences if necessary and using a small perturbations we obtain block sequences $(u_i')$ and $(v_i')$ of the form $u_i'=u_i+\de_ie_{m_{i}}$, $v'_i=v_i+\de_i e_{m_{i}}$, for some $(m_{i})\subset\N$ with $m_{i}=\minsupp u'_i=\minsupp v'_i$ for each $i$ and small $(\de_i)\subset (0,1)$, which are equivalent to $(u_i)$ and $(v_i)$ respectively and satisfy still the assertion of Lemma \ref{theta3} for $2M_k-3$. Then construct on these sequences two  Tsirelson $(2M_k-3,\e_k)$-averages with averaging trees as in Lemma \ref{ris} with equal systems of weights, errors and coefficients, obtaining $x_k$ and $y_k$. 

Now we proceed to the proof of Lemma \ref{ris}.
\bp Notice first that by Lemmas \ref{theta1} and \ref{theta3} we have estimation
$$
\theta^{2M_k-3}/4\leq\norm[x_k]\leq 5\theta_1^{-2}\theta^{2M_k-3}, \ \ \ \ k\in\N\,,
$$
and the same estimation for $\norm[y_k]$, $k\in\N$.

We show first that $(y_k/\norm[y_k])_k$ dominates $(x_k/\norm[x_k])_k$. Let $x=\sum_kd_kx_k/\norm[x_k]$ be of norm 1, with $(d_k)\subset [0,1]$, and take its norming functional $f$ with a tree-analysis $(f_\al)_{\al\in\mt}$. Let $y=\sum_kd_ky_k/\norm[y_k]$. By Fact \ref{ave0} we can assume with error $\e$ that $\ord(\al)< M_{k+1}/4\leq M_{k+1}-4$ for any $\al\in\mt$ with $f_\al(x_k)\neq 0$. For any $k>1$ let
$$
\mc{E}_k=\{\al\in\mt:\ f_\al \text{ begins at }x_k \text{ and has a sibling beginning before } x_k\}.
$$
By our reduction $\ord(\al)<M_k-4$ for any $\al\in \mc{E}_k$, $k\geq 2$. We replace in the tree-analysis of $f$ each functional $f_\al$, $\al\in \mc{E}_k$, by two functionals $g_\al=f_\al|_{\supp x_k}$ and $k_\al=f_\al-g_\al$, obtaining a tree-analysis of a functional $g$ on the space $X_2=T[(\mc{S}_n[\mc{A}_2],\theta_n)_n]$, which by  Lemma \ref{x2} is 3-isomorphic to $X$.

Notice that $(g_\al)_{\al\in\mc{E}_k,k\geq 2}$ have pairwise disjoint supports and $(\bigcup_{\al\in\mc{E}_k}\supp g_\al)\cap\supp x_k=\supp f\cap \supp x_k$, hence $f|_{\supp x_k}=\sum_{\al\in\mc{E}_k}t(\al)g_\al$. For each $k\geq 2$ consider the set $J_k=\{i: \text{ some }g_\al \text{ begins at } x_{k,i}^{M_k}\}$. Notice that by our reduction $(g_\al)_{\al\in\mc{E}_k}$ is $\mc{S}_{M_k-4}$-allowable, thus $(x_{k,i}^{M_k})_{i\in J_k}$ is $\mc{S}_{M_k-4}$-admissible and recall that $x_k$ is an $(M_k-3,\e_k)$-average of $(x_{k,i}^{M_k})$. Let $g_\al'$, $\al\in \mc{E}_k$, be the restriction of $g_\al$ to $\cup_{i\not\in J_k}\supp x_{k,i}^{M_k}$. Then we have the following estimation
\begin{align*}
f(x)&=\frac{d_1}{\norm[x_1]}f(x_1)+\sum_{k\geq 2}\frac{d_k}{\norm[x_k]}f(x_k)\\
&\leq\frac{d_1}{\norm[x_1]}f(x_1)+\sum_{k\geq 2}\frac{d_k}{\norm[x_k]}\sum_{\al\in\mc{E}_k}t(\al)g'_\al(x_k)+\sum_k\frac{d_k}{\norm[x_k]}\sum_{i\in J_k}a_{k,i}^{M_k}\norm[x_{k,i}^{M_k}]\\
&\leq\frac{d_1}{\norm[x_1]}f(x_1)+\sum_{k\geq 2}\frac{d_k}{\norm[x_k]}\sum_{\al\in\mc{E}_k}t(\al)g'_\al(x_k)+4\sum_k\e_k\theta^{-2M_k+3}\\
&\leq\frac{d_1}{\norm[x_1]}f(x_1)+\sum_{k\geq 2}\frac{d_k}{\norm[x_k]}\sum_{\al\in\mc{E}_k}t(\al)g'_\al(x_k)+\e\,.
\end{align*}
Fix $k\geq 2$. Notice that by definition the set $\{g'_\al:\ g'_\al(x_{k,i}^{M_k})\neq 0\}$ restricted to the support of $x_{k,i}^M$ is $\mc{S}_1$-allowable for any $i$. Therefore by Lemma \ref{theta1.7} we pick suitable partition $(V_{\al})_{\al\in\mc{E}_k}$ of nodes $(x_{k,i}^0)_i$ with $\minsupp x_{k,\min V_\al}^0\geq \minsupp g'_\al$ for each $\al\in\mc{E}_k$ and applying Lemma  \ref{theta3} we have
\begin{align*}
\sum_{\al\in\mc{E}_k}t(\al)g'_\al(x_k)&\leq C\sum_{\al\in\mc{E}_k}t(\al)\norm[\sum_{i\in V_{\al}}a^0_{k,i}e_{\minsupp x_{k,i}^0}]_{T[\mc{S}_1,\theta]}+C\e_k\\
 &\leq C\sum_{\al\in\mc{E}_k}t(\al)\norm[\sum_{i\in V_{\al}}a^0_{k,i}e_{\minsupp y_{k,i}^0}]_{T[\mc{S}_1,\theta]}+C\e_k\\
& \leq 2C\sum_{\al\in\mc{E}_k}t(\al)\norm[\sum_{i\in V_{\al}}a^0_{k,i}y_{k,i}^0]+C\e_k\\
& \leq 2C\sum_{\al\in\mc{E}_k}t(\al)h_\al (y_k)+C\e_k\,,
\end{align*}
where $h_\al$ is a norming functional on $X$ with $h_\al(y_k)=\norm[\sum_{i\in V_\al}a^0_{k,i}y_{k,i}^0]$ and $\minsupp h_\al\geq \minsupp x_{k,\min V_\al}^0\geq \minsupp g'_\al$ for each $\al\in \mc{E}_k$.

We modify the tree-analysis of $g$, replacing each node $g_\al$, $\al\in\mc{E}_k$, $k\geq 2$, by the functional $h_\al$. As $\minsupp h_\al\geq \minsupp g_\al$ for each $\al$, we obtain a tree-analysis of some norming functional $h$ on $X_2$. We compute, by Lemma \ref{x2} and above estimations including the estimation on the norms of $(x_k)_k$ and $(y_k)_k$,
\begin{align*}
1=f(x)&\leq d_1+\sum_{k\geq 2}\frac{d_k}{\norm[x_k]}\sum_{\al\in\mc{E}_k}t(\al)g_\al(x_k)+\e\\
& \leq d_1+40C\theta^{-2}\sum_{k\geq 2}\frac{d_k}{\norm[y_k]}\sum_{\al\in\mc{E}_k}t(\al)h_\al(y_k)+4C\sum_{k\geq 2}\frac{\e_k}{\theta_{2M_k-3}}+\e\\
&\leq d_1+40C\theta^{-2}h(\sum_{k\geq 2}\frac{d_k}{\norm[y_k]}y_k)+3C\e\\
&\leq 121C\theta^{-2}\norm[y]+1/2\,,
\end{align*}
which means that $(y_k/\norm[y_k])_k$ dominates $(x_k/\norm[x_k])_k$. Since the conditions are symmetric, the opposite domination follows analogously.
\ep
\section{Strictly singular non-compact operators}
\subsection{Spaces defined by families $(\mc{A}_n)_n$}
As in mixed Tsirelson spaces defined by Schreier families the crucial tool will  be formed by $\ell_p-$averages.
\bd A vector $x\in X$ is called a $C-\ell_r-$\textit{average} of length $m$, for $r\in
[1,\infty]$, $m\in\N$ and $C\geq 1$ if $x=\sum_{i=1}^mx_i/\nrm{\sum_{i=1}^mx_i}$ for some
normalized block sequence $(x_n)_{n=1}^m$ which is $C$-equivalent to the unit vector
basis of $\ell_r^m$.
\ed
\bd \cite{s2} Let $X$ be a Banach space with a basis $(e_n)$. Then $X$ is in 
\bnum
 \item Class 1, if every normalized block sequence in $X$ has a subsequence equivalent to some subsequence of $(e_n)$. 
 \item Class 2, if each block sequence has further normalized block sequences $(x_n)$ and $(y_n)$ such that the map $x_n\mapsto y_n$ extends to a bounded strictly singular operator between $[x_n]$ and $[y_n]$. 
\enum
\ed
T. Schlumprecht asked if any Banach space contains a subspace with a basis which is either of Class 1 or Class 2 and gave some sufficient condition (Thm. 1.1 \cite{s2}) for the existence of strictly singular non-compact operator in the space. 
\bt\cite{s2}\label{s2} Let $(x_n)$ and $(y_n)$ be two normalized basic sequences generating spreading models $(u_n)$ and $(v_n)$ respectively. Assume that $(u_n)$ is not equivalent to the u.v.b. of $c_0$ and $(u_n)$ strongly dominates $(v_n)$, i.e. 
$$
\norm[\sum_{i=1}^\infty a_iv_i]\leq \max_{n\in\N}\de_n\max_{\# F\leq n}\norm[\sum_{i\in F} a_iu_i]
$$
for some sequence $(\de_n)$ with $\de_n\searrow 0$, $n\to\infty$. Then the map $x_n\mapsto y_n$ extends to a bounded strictly singular operator between $[x_n]$ and $[y_n]$. 
\et
\bt\label{p-sp} Let $X=T[(\mc{A}_n, \frac{c_n}{n^{1/q}})_n]$ be a regular $p-$space, with $p\in [1,\infty)$. Then 
\bnum
 \item if $\inf_n c_n>0$, then $X$ is saturated with subspaces of Class 1.
 \item if $c_n \to 0$, $n\to\infty$, then $X$ is in Class 2. 
\enum
\et
\bp PART (1). We show that any block subspace of $X$ contains a normalized block sequence $(u_s)_s$ with the following "blocking principle": any normalized block sequence $(y_j)_j$ is equivalent to any $(u_{k_j})_j$, with $y_j<u_{k_{j+1}}$ and $u_{k_j}<y_{j+1}$.  
It follows that the subspace $[(u_s)]$ is sequentially minimal.. 

By Prop. 2.10 \cite{mp} any block subspace of $X$ contains an $\ell_p$-asymptotic subspace of $X$. Let $W$ be such $\ell_p$-asymptotic subspace, spanned by a normalized block sequence $(w_k)_k$. Let $C$ be the asymptotic constant of $W$, i.e. any normalized block sequence $(z_i)_{i=1}^n$ with $z_1>n$ in $W$ is $C$-equivalent to the u.v.b. of $\ell_p^n$.

For any block subspace $Y$ of $X$ spanned by normalized block sequence $(y_n)$ let $\norm[\sum_na_ny_n]_{Y,\infty}=\sup_{n\in\N}|a_n|$. 

Fix two strictly increasing sequences of integers $(m_n)_n\subset\N$ and $(N_j)_j\subset\N$ and take normalized block sequences $(v_n)_n$ of $(w_k)_k$ and $(u_j)_j$ of $(v_n)_n$ such that 
\bnum
 \item $v_n>m_n$ in $W$ for any $n$, 
 \item for any $y\in [(v_i)_{i>n}]$ we have $\norm[y]_{W,\infty}<1/(8m_n^5)$,  for any $n$, 
 \item $u_j>N_j$ in $V=[(v_n)_n]$ for any $j$, 
 \item for any $y\in [(u_i)_{i>j}]$ we have $\norm[y]_{V,\infty}<1/(8N_j^5)$, for any $j$, 
 \item $\sqrt[p]{N_j}\geq C2^{j+7}$ for any $j$
 \item $N_j\theta_{m_n}<1/2^{n+5}$ for any $n\geq j$ (in particular $m_n\geq N_j$ for any $n\geq j$)
 \item $\theta_{m_n}\sum_{i<n}\# \supp v_i<1/2^{n+5}$ for any $n$
\enum
Notice that every vector $y\in [(v_i)_{i>n}]$ is an $2C-\ell_p$-average of length $m_n$ of some normalized block sequence $(y_i)_{i=1}^{m_n}$ of $(w_k)_k$.  Indeed, by Claim 3.8 \cite{mp} and condition (2)  split $y$ into $(Fy_i)_{i=1}^{m_n}$ with almost equal norm and obtaining by condition (1) and  $\ell_p$-asymptoticity of $W$ that $y$ is a suitable average. The same holds in $V$: every vector $y\in [(u_i)_{i>j}]$ is an $2C-\ell_p$-average of length $N_j$ of some normalized block sequence $(y_i)_{i=1}^{N_j}$ (block with respect to $(v_n)_n$).

We show that under such conditions we can prove the above Theorem repeating the proof of Theorem 3.1 \cite{mp}. We consider any normalized block sequence $(y_j)$ of $(u_j)$ and as $(z_j)$ we take $(u_{k_j})$ with $y_j<u_{k_{j+1}}$ and $u_{k_j}<y_{j+1}$. By the above observation $y_j=(y^j_1+\dots+y^j_{N_j})/\nrm{y^j_1+\dots+y^j_{N_j}}$ and $u_{k_j}=(u^j_1+\dots+u^j_{N_j})/\nrm{u^j_1+\dots+u^j_{N_j}}$, where $(y_j^i)_{i=1}^{N_j}$ and $(u_j^i)_{i=1}^{N_j}$ are normalized block sequences with respect to $(v_j)_j$. Notice that $(N_j)$ are big enough by condition (5). We again use the above observation obtaining that each $y_j^i$ and $v_j^i$ is an $\ell_p$-average of a block sequence, of $(w_k)_k$, of suitable length with parameters satisfying the assertion of a version of Lemma 3.2 \cite{mp} for  $C$-averages instead of 2-averages (by conditions (6) and (7)). Therefore  repeating the proof of Theorem 3.1 \cite{mp} we obtain uniform equivalence of $(y_j)$ and $(u_{k_j})$ and hence "blocking principle" stated above. 

PART (2). Fix  a block subspace $Y$ of $X$. By Theorem 2.9 \cite{mp} $p$ is in Krivine set of $Y$. 
Take finite normalized block sequences $(y_{i})_i$ such that for some $(m_{i})_i\subset\N$ 
\bnum
 \item each $y_i$ is $2-\ell_p-$averages of length $N_i\geq (2m_i)^p$,
 \item $\theta_{m_i}\sum_{j<n}\# \supp y_j\leq 1/2^{i+5}$ for any $i$, 
 \item $2^{i+5}\theta_{m_i}\to 0$, $i\to \infty$.
\enum
Passing to a subsequence we can assume that $(y_i)$ generates a spreading model $(v_i)$. 
\bl \label{l3.5} The spreading model $(v_i)$ is strongly dominated by the u.v.b. of $\ell_p$. 
\el
\bp  Take $k\in \N$ and $(a_i)_{i=1}^N\in c_{00}$ with $\norm[(a_i)]_\infty\leq 1/k^2$ and $\norm[(a_i)]_{\ell_p}=1$. 
Choose $M$ by (3) in definition of $(y_{i})$ with  $N\theta_{m_{i+M}}\leq 1/2^{i+M+5}$ for any $i$ and $1/2^M\leq 1/k$. We have $\norm[\sum_{i=1}^Na_iv_i]\leq 2\norm[\sum_{i=1+M}^{N+M}\tilde{a}_iy_{i}]$, where $\tilde{a}_{i+M}=a_i$, $i=1,\dots,N$.

Take a norming functional $f$ with a tree-analysis $(f_t)_{t\in \mt}$ and $\supp f\subset\supp y$, where $y=\sum_{i=1+M}^{N+M}\tilde{a}_iy_{i}$. By Lemma 2.5 \cite{mp} up to multiplying by 36 we can assume that for any $f_t$ and  $y_{i}$ we have either  $\supp f_t\subset y_{i}$, $\supp f_t\supset \supp y_{i}\cap\supp f$ or $\supp f_t\cap\supp y_{i}=\emptyset$. We say that $f_t$ covers $y_{i}$, if $t$ is maximal in $\mt$ with $\supp f_t\supset\supp y_{i}\cap \supp f$.

Let $A=\{t\in T:\ f_t$ covers some $y_{i}$\}. Given any $t\in A$ let $I_t=\{ i=1+M,\dots,N+M:\ f_t$ covers $y_{i}\}$. Let $\theta_{m_t}$ be the weight of $f_t$. If $m_t>m_{i}$ for  some $i\in I_t$ let $i_t$ be the maximal element of $I_{t}$ with this property. Otherwise let $i_t=0$. 

For any $i\in I_t$ let $J_i=\{s\in\suc (t):\ \supp f_s\subset \supp y_{i}\}$. By Lemma 2.8 \cite{mp} we have $\sum_{s\in J_i}f_s(y_{i})\leq 8 (\# J_i)^{1/q}$ for each $i\in I_t, i>i_t$.

First let $L_t=\{i\not\in I_t:\ \supp y_{i}\cap \supp f\subset \supp f_t\}$. Notice that for any $i\in L_t$ there is some $f_{t_i}$ - successor of $f_t$ so that $\supp y_{i}\cap \supp f\subset \supp f_{t_i}$. Hence
$$
f_t(\sum_{i\in L_t}\tilde{a}_iy_{i})\leq \theta_{m_{i_t}}(\sum_{i\in L_t}f_{t_i}(\tilde{a}_iy_{i}))\leq N\theta_{m_{i_t}}\leq 1/2^{i_t+2}\,.
$$
Thus $f(\sum_{t\in A,i\in L_t}y_{i})\leq 1/2^M$ and we erase this part for all $t$ with error $\leq 1/k$.
Notice that by condition (2) in choice of $(y_i)$ we have
$$
f_t(\sum_{i\in I_t, i<i_t}y_i)\leq \theta_{m_{i_t}}\sum_{i<i_t}\# \supp y_{i} \leq 1/ 2^{i_t+2}\,, 
$$
so we can again erase this part for all $t$ with error $1/k$.

Let $g$ be the restriction of $f$ to $\cup_{t\in A}\supp y_{i_t}$ and $h=f-g$. First we consider $g(y)=\sum_{t\in A}t(f_t) \tilde{a}_{i_t}f_t(y_{i_t})$. Let $B=\{t\in A:\  \ord(f_t)\leq k\}$, hence $\# B\leq k$. Then $\sum_{t\in B}\tilde{a}_{i_t}f_t(y_{i_t})\leq \# B/k^2\leq 1/k $, hence we can erase this part with error $1/k$. Notice that $\sum_{t\in A\setminus B}\frac{1}{\ord(f_t)^{1/q}}e^*_{i_t}$ is a norming functional on $\ell_p$, hence 
$$
\sum_{t\in A\setminus B}\tilde{a}_{i_t}t(f_t)f_t(y_{i_t})\leq \sum_{t\in A\setminus B}\tilde{a}_{i_t}\frac{c_{\ord(f_t)}}{(\ord(f_t))^{1/q}}\leq \max_{n\geq k}c_n\norm[(\tilde{a}_{i_t})_{t\in A\setminus B}]_{\ell_p}\leq \max_{n\geq k}c_n\,.
$$
We consider $h(y)=\sum_{t\in A}\sum_{i\in I_t, i> i_t}\tilde{a}_i\sum_{s\in J_i}t(f_s)f_s(y_{i})$. Let $D=\{s\in J_i, i\in I_t, i> i_t, t\in A: \ \ord (f_s)\leq k\}$. Then 
$$
\sum_{t\in A}\sum_{i\in I_t, i> i_t}\sum_{s\in J_i\cap D}\tilde{a}_if_s(y_{i})\leq \# D/k^2\leq 1/k\,,
$$ 
and we again erase this part with error $1/k$. For any $i\in I_t, i> i_t$ for some $t\in A$ we let $r_i=\ord(f_t)m_t$ and compute, using H\"older inequality,
\begin{align*}
 \sum_{t\in A}\sum_{i\in I_t, i> i_t}\sum_{s\in J_i\setminus D}\tilde{a}_it(f_s)f_s(y_{i})\leq &\sum_{t\in A}\sum_{i\in I_t, i> i_t}\tilde{a}_i8 (\# J_i)^{1/q}\theta_{r_i}\\
 &\leq 8\max_{n\geq k}c_n\sum_{t\in A}\sum_{i\in I_t, i> i_t}\tilde{a}_i \frac{(\# J_i)^{1/q}}{r_i^{1/q}}\\
 &\leq 8\max_{n\geq k}c_n \norm[(\tilde{a}_i)_{i\in I_t,i> i_t,t\in A}]_{\ell_p}\leq 8\max_{n\geq k}c_n\,.
\end{align*}
We put all the estimates together obtaining 
$$
f(y)\leq 36(9 \max_{n\geq k}c_n+4/k)\,.
$$
Therefore we proved that $\Delta_{\e}=\sup\left\{ \norm[\sum_{i\in \N}a_iv_i]:\ \sup_{i\in \N}|a_i|\leq \e,\ \norm[(a_i)_{i\in \N}]_{\ell_p}=1\right\}$ converges to zero, as $\e\to 0$. By Lemma 2.4 \cite{s2} there are some $(\delta_n)_n\subset (0,\infty)$ with $\delta_n\searrow 0$ such that for any $(a_i)_i\in c_{00}$
$$
\norm[\sum_{i}a_iv_i]\leq \max_{n\in\N}\delta_n \max_{\# F\leq n}\norm[(a_i)_{i\in F}]_{\ell_p}\,,
$$
which ends the proof of Lemma.
\ep
We continue the proof of Theorem \ref{p-sp}. By the proof of Thm 2.9 \cite{mp}, $p$ is in the Krivine set of $Y$ in Lemberg sense \cite{l}, i.e. for any $n$ there is a normalized block sequence $(x^{(n)}_i)_i\subset Y$ generating spreading model $(u^{(n)}_i)_i$ such that $(u^{(n)})_{i=1}^n$ is 1-equivalent to the u.v.b. of $\ell_p^n$. 

Pick $(m_n)_n$ such that $\delta_{m_n}\leq 1/4^n$. Apply Prop. 3.2 \cite{aost} to constants $C_n=2^n$, $n\in\N$ and normalized block sequences $(x_i^{(m_{n})})_i$ generating spreading models $(u_i^{(m_{n})})_i$. We obtain thus a seminormalized block sequence $(x_i)$ generating spreading model $(u_i)_i$ which $C_n$ dominates $(u_i^{(m_{n})})_i$ for any $n\in\N$. By Lemma \ref{l3.5} we obtain
\begin{align*}
 \norm[\sum_{i}a_iv_i]&\leq \max_{n\in\N}\delta_n \max_{\# F\leq n}\norm[(a_i)_{i\in F}]_{\ell_p}\\
&\leq \max_{n\in\N}\delta_{m_{n}} \max_{\# F\leq m_{n+1}}\norm[(a_i)_{i\in F}]_{\ell_p}\\
&\leq \max_{n\in\N}1/4^n \max_{\# F\leq m_{n+1}}\norm[\sum_{i\in F}a_iu_i^{(m_{n+1})}]\\
&\leq \max_{n\in\N}C_{n+1}/4^n \max_{\# F\leq m_{n+1}}\norm[\sum_{i\in F}a_iu_i]\\
&\leq \max_{n\in\N}2/2^n \max_{\# F\leq m_{n+1}}\norm[\sum_{i\in F}a_iu_i]\,.
\end{align*}
Notice that $(u_i)$ is not equivalent to $c_0$, thus by Theorem \ref{s2} we finish the proof. 
\ep
In \cite{gas2} the construction of non-compact strictly singular operators was based on $c_0$-spreading model of higher order in the dual space. However this method does not follow straightforward in case of $p-$spaces, as the observation below shows. We consider the Schlumprecht space $S=T[(\mc{A}_n,\frac{1}{\log_2(n+1)})_n]$ introduced in \cite{s1}. In \cite{kl} it was shown that $S$ contains a block sequence generating $\ell_1$-spreading model. 
\bpr \label{c-0} Consider the sequence $(y_k)$ generating $\ell_1$-spreading model constructed in \cite{kl}, $y_k=\sum_{m=1}^kv_{k,m}$, $k\in\N$. Take any block sequence $(y_k^*)\subset S^*$ so that $y^*_k(y_l)=\de_{l,k}$. Then the sequence $(y_k^*)$ does not generate $c_0$-spreading model.
\epr
\bp We can assume that $\supp y_k^*=\supp y_k$, $k\in\N$. Consider two cases:

\textsc{CASE 1}. There is $m_0\in\N$, $\de>0$ and an infinite $K\subset\N$ with $\abs[y^*_k(\sum_{m=1}^{m_0}v_{k,m})]\geq \de$ for any $k\in K$.

Let $z^*_k$ be the restriction of $y^*_k$ to the support of $\sum_{m=1}^{m_0}v_{k,m}$, $k\in K$. Then $(z^*_k)_{k\in K}$ is a seminormalized block sequence in $S^*$, majorized by $(y_k^*)_{k\in K}$. Since by the form of $(v_{m,k})$ the length of $\supp (\sum_{m=1}^{m_0}v_{k,m})$ is constant, we can pick some subsequence $(z^*_k)_{k\in L}$ of $(z_k^*)_{k\in K}$ consisting of, up to controllable error, equally distributed vectors. As the u.v.b. in $S$ is subsymmetric, the same holds for $(z^*_k)_{k\in L}$, thus $(z^*_k)_{k\in L}$ is equivalent to spreading model generated by itself. It follows that $(y^*_k)$ cannot generate $c_0$-spreading model.

\textsc{CASE 2}. If the first case does not hold, pick increasing $(N_j)\subset\N$ so that
$$
\left|y^*_{N_j}\left(\sum_{m=1}^{N_{j-1}}v_{N_j,m}\right)\right|\leq 1/2^j\,.
$$
Consider the norm of vectors $z_j^* = y_{N_1}^* + \dots+ y_{N_j}^*$. Put
$$
x_{N_1}= y_{N_1},\ \ \ x_{N_j}=\sum_{m=N_{j-1}+1}^{N_j}v_{N_j,m},\ \ \ j>1\,.
$$
By the choice of $(N_j)$ we have $y^*_{N_j}(x_{N_j})\geq 1 - 1/2^j$. 

We estimate the norm of $x_j = x_{N_1} + \dots + x_{N_j}$. We can assume at the beginning that $(N_j)$ was chosen to increase fast enough so that  $(x_{N_j})$ is $D$-equivalent to the unit basis of $S$ (see Remark 5, Lemma 2 \cite{kl}). Therefore $\nrm{x_j}\leq Dj/f(j)$.

By the choice of $(N_j)$ and definition of $x_{N_j}$ we have $z_j^*(x_j) \geq j-1$. Hence 
$$
\nrm{z_j^*}\geq z_j^*(x_j)/\nrm{x_j}\geq f(j)(j-1)/Dj \geq f(j)/2D\,.
$$
Notice that the same scheme works if we replace $N_1,\dots,N_j$ by any $N_{n_1},\dots,N_{n_j}$ in definition of $z_j$, hence no subsequence of $(y_k^*)$ can produce a $c_0$-spreading model.
\ep
\subsection{Spaces defined by families $(\mc{S}_n)_n$}
Regarding the existence of strictly singular operators from subspaces of mixed  Tsirelson spaces we prove the following result, which is in  ''localization''  of Schlum\-precht result in mixed Tsirelson spaces. First recall the definition of a higher order $\ell_1$-spreading models. 

\bd We say that a normalized basic sequence $\xn$ in a Banach space generates an $C-\ell_1^\al$-spreading model, $\al<\omega_1$, $C\geq 1$, if for any $F\in\mc{S}_\al$ the sequence $(x_n)_{n\in F}$ is $C-$equivalent to the u.v.b. of $\ell_1^{\# F}$. In case of $\al=1$ we obtain the classical $\ell_1$-spreading model.
\ed
We recall that $[M]$, $M\subset\N$, denotes the family of all infnite subsequences of $M$, $[M]^<$ - the family of all finite subsequences of $M$. 
\begin{theorem}\label{thm3.7}
 Let $X=T[(\mathcal{S}_{n},\theta_{n})_{n}]$  or $T_M[(\mathcal{S}_{n},\theta_{n})_{n}]$ be a regular (modified) mixed Tsirelson space. If $X$ contains a  block sequence $(y_n)$ generating $\ell_{1}^{\omega}$-spreading model then there are a subspace $Y\subset [(y_n)]$ and  a strictly singular operator   $T:Y\to X$.
\end{theorem}
We recall that in \cite{lt} it was proved that if a regular sequence $(\theta_n)$ satisfies $\lim_m\limsup_n\frac{\theta_{m+n}}{\theta_n}>0$
then the mixed Tsirelson space $X=T[(\mc{S}_n,\theta_n)_n]$ is subsequentially minimal if and only if any block subspace of $X$ admits an $\ell_1^\omega$-spreading model, if and only if any block subspace of $X$ has Bourgain $\ell_1-$index greater than $\omega^\omega$. These conditions hold in particular if $\sup\theta_n^{1/n}=1$ \cite{m}. In \cite{klmt} analogs of these results were studied in the partly modified setting.

To prove the theorem  we first define an index measuring the best constant of the $\ell^{{\alpha}}_{1}$-spreading models generated by subsequences  of a given sequence. Let   $\vec{x}:=\xn$ be a normalized block sequence. We set
$$
\delta_{\alpha}(\vec{x})=\sup\{\delta>0: \mbox{$\exists M\in [\N]$ such that $(x_{n})_{n\in M}$ generates $\delta-\ell^{\alpha}_{1}$ spr. model} \}\,.
$$
The following properties of $\delta_{\alpha}(\vec{x})$ follows readily from the definition.
\begin{enumerate}
\item[ a)]  $\delta_{\alpha}(\xn)=\delta_{\alpha}((x_{n})_{n\geq n_{0}})$  for all $n_{0}\in\N$.
\item[b)]  $\delta ((x_{n})_{n\in M})\leq \delta_{\alpha}  (\xn)$  for all $M\in [\N]$.
\item[c)]  $(\delta_{\alpha}(\vec{x}))_{\alpha<\omega_{1}}$ is non-increasing family.
\end{enumerate}
By standard arguments we may stabilize  $\delta_{\alpha}(\vec{x})$. Namely passing to a subsequence    we may assume that $\delta_{\alpha}(\xn)=\delta_{\alpha}((x_{n})_{n\in M})$ for every $M\in 
[\N]$.

By Bourgain's  $\ell_{1}-index$ it follows that  $\delta_{\alpha}(\xn)>0$ countable many $\alpha's$, enumerate them as $(\al_n)_n$. In particular for  an asymptotic $\ell_{1}$ space it  follows that $\delta_{n}(\vec{x})>0$ for all  $n\in\N$.

Inductively  we  choose  $M_{1}\supset M_{2}\supset\dots$ infinite subsets of $\N$  such that
\begin{align*}
\delta_{\alpha_{n}}((x_{n})_{n\in M_{n}} )=\delta_{\alpha_{n}} (x_{n})_{n\in  L}\,\,\forall   L\in [M_{n}]\,.
\end{align*}
We define  the family 
$$
\mathcal{F}_{2\delta_{\alpha_{n}}(\vec{x})}=\{A\in [\N]^{<}: \mbox{$\exists x^{*}\in B_{X^{*}}$ with $x^{*}(x_{i})>2\delta_{\alpha_{n}}(\xn)$ for all $i\in A$} \}.
$$
By I. Gasparis theorem \cite{gas}  there exists $N\in[M_{n}]$ such that 
$$
\mbox{either    $\mathcal{S}_{\alpha_{n}}\cap [N]\subset \mathcal{F}_{2\delta_{\alpha_{n}} }$   or $\mathcal{F}_{2\delta_{\al_n}}\cap [N] \subset\mathcal{S}_{\alpha_{n}}$. }
$$
In the first case by 1-unconditionality of the basis it follows that  $(x_{n})_{n\in N}$ and hence $(x_{k})_{k\in M_{n}}$ contains a subsequence  which generates $2\delta_{\alpha_{n}}-\ell_{1}^{\alpha_{n}}$-spreading model,  a contradiction.  So additionally we may assume that there exists $M_{n}\in [M_{n-1}]$ with
\begin{align}\label{eq1}
\mathcal{F}_{2\delta_{\alpha_{n}}} (M_{n})\subset \mathcal{S}_{\alpha_{n}} ,
\\
\mathcal{S}_{\alpha_{n-1}}\cap \{m_{n}, m_{n}+1,\dots\}\subset\mathcal{S}_{\alpha_{n}}\,.
\label{eq1a}
\end{align}
Let $M=(m_{i})_{i} $ be a diagonal set. Passing to a subsequence we may assume that $\sum_{n}n\delta_{\alpha_{n}}<0.25$. Let $\norm[\sum_{i}a_{i}x_{m_{i}}]=1$ and let  $x^{*}\in B_{X^{*}} $ such that  $\sum_{i}a_{i}x^{*}(x_{m_i})=1$. By the unconditionality  we may assume that  $x^{*}(x_{m_i})\geq 0$ for every $i$. Let  $2\delta_{\alpha_{0}}=1$ and
$$
F_{k}=\{i: x^{*}(x_{m_{i}})\in
(2\delta_{\alpha_{k}},2\delta_{\alpha_{k-1}}] \} 
$$
and $F_{k}^{1}=F_{k}\cap\{1,\dots,k-1\}$,
$F_{k}^{2}=F_{k}\cap\{k,k+1,\dots\}$.

From  \eqref{eq1},\eqref{eq1a}  we get $F_{k}^{2}\in\mathcal{S}_{\alpha_{k}}\cap \{k,k+1,\dots\}=\mathcal{G}_{k}$.    It follows
\begin{align*}
\norm[\sum_{i}a_{i}x_{m_{i}}]&=\sum_{i}a_{i}x^{*}(x_{m_{i}})=\sum_{k=1}^{\infty}\sum_{i\in  F_{k}}a_{i}x^{*}(x_{m_{i}})
\\
& = \sum_{k=1}^{\infty}\left(\sum_{i\in F_{k}^{1}}a_{i}x^{*}(x_{m_{i}})+\sum_{i\in F_{k}^{2}}a_{i}x^{*}(x_{m_{i}})\right)
\\
&\leq
\sum_{k=2}^{\infty}2\delta_{\alpha_{k-1}}(k-1)\max_{i}\abs[a_{i}]+ 
\sum_{k=1}^{\infty}2\delta_{\alpha_{k-1}}\sum_{i\in F_{k}^{2}}\abs[a_{i}]
\\
&\leq 0.5\norm[\sum_{i}a_{i}x_{m_{i}}]+ \sum_{k=1}^\infty2\delta_{\alpha_{k-1}}\sup_{F\in\mathcal{G}_{k}}\sum_{i\in F}\abs[a_{i}]\,,
\end{align*}
and therefore
$\norm[\sum_{i}a_{i}x_{m_{i}}]\leq 4 \sum_{k=1}^\infty\delta_{\alpha_{k-1}}\sup_{F\in\mathcal{G}_{k}}\sum_{i\in F}\abs[a_{i}]. $

So we have the following
\begin{equation} \label{eq3}
 \norm[\sum_{i}a_{i}x_{m_{i}}]\leq 4 \sum_{k=1}^\infty\delta_{\alpha_{k-1}}\sup_{F\in\mathcal{G}_{k}}\sum_{i\in F}\abs[a_{i}]\,\, \textrm{for all}\,\, (a_{i})_{i},
\end{equation}
where $\mathcal{G}_{k}=\mathcal{S}_{\alpha_{k}}\cap\{k,k+1,\dots\}$.
\begin{proof}[Proof of the Theorem \ref{thm3.7}]
Let $\vec{e}=\xn[e]$  be the basis  of $X$.  Using that for every $j\in\N$ and every  $\sum_{i\in F}a_{i}e_{i}$ special convex combination of the basis it holds 
$$
\theta_{n}\leq \norm[\sum_{i\in F}a_{i}e_{i}]\leq 2\theta_{n}\,,
$$
see \cite{ad2,adkm}. It follows readily that  $\delta_{n}(\vec{e})\in [\theta_{n},2\theta_{n}]$ and $\delta_{\omega}=0$.

Since the space $X$ contains  a a block sequence $\xn[y]$ generating $\ell_{1}^{\omega}$-spreading model  it follows that
$$
\norm[\sum_{i}a_{i}y_{i}]\geq c\sum_{i\in F}\abs[a_{i}]\,\,\,\forall n\in\N, F\in\mathcal{S}_{n}\cap\{n,n+1,\dots\}. 
$$
By the previous reasoning we pick a $M=(m_{i})\in [\N]$ and a  sequence 
$\alpha_{k}\nearrow \omega$ such
that $\sum_k k\delta_{\al_k}<\infty$ and \eqref{eq3} holds. Setting $M=\sum_{k}\theta_{\alpha_{k-1}}$  we have
\begin{align*}
\norm[\sum_{i}a_{i}e_{m_i}] &\leq 
8\sum_{k} \theta_{\alpha_{k-1} } \sup_{F\in\mathcal {G}_{k}}\sum_{i\in
F}\vert a_i\vert   
\\
&\leq 
\frac{8M}{c}\sup_{k}c\sup_{F\in \mathcal{G}_{k}}\sum_{i\in F}\vert
a_{i}\vert 
\\
&\leq  \frac{8M}{c}\norm[\sum_{i}a_{i}y_{i}]\,.
\end{align*}
It follows that  the operator  extending the mapping $y_{n}\to x_{m_{n}}$  factors through a $c_0$-saturated space and hence is strictly singular.
\end{proof}
\subsection{Remarks and questions}
As a corollary to  Theorem \ref{p-sp}, part (1), we obtain that the (non-modified) Tzafriri space $Y$ has an asymptotic $\ell_2$ subspace $Z$ which satisfies a blocking principle in the sense of \cite{cjt}. The only known spaces with a blocking principle so far were similar to $T$, $T^*$ and their variations. The two major ingredients used in \cite{cjt} for proving the minimality of $T^*$ are the blocking principle and the saturation with $\ell_\infty^n$'s. It is shown in \cite{jko} that Tzafriri space $Y$ contains uniformly $\ell_\infty^n$'s. It is not known whether $Y$ is uniformly saturated with $\ell_\infty^n$'s. In the opposite direction, we do not know if $Z$ contains a convexified Tsirelson space $T^{(2)}$ (which is equivalent to its modified version).

In 1977 Altshuler \cite{alt} (cf. e.g. \cite{lt1}) constructed a Banach space with a symmetric basis which contains no $\ell_p$ or $c_0$, and all its symmetric basic sequences are equivalent. In 1981 C. Read \cite{read} constructed a space with, up to equivalence, precisely two symmetric bases. More precisely, Read proved that any symmetric basic sequence in his space CR is equivalent either to the u.v.b. of $\ell_1$ or to one of the two symmetric bases of CR. 
A careful look at the papers of Altshuler and Read shows that their proofs work similarly for the more general case of all subsymmetric basic sequences. This observation leads to the following questions:

\textbf{Question 1.} Does there exist a space in which all subsymmetric basic sequences are equivalent to one basis, and that basis is not symmetric?

We remark that Altshuler's space has a natural subsymmetric version but we do not know if it satisfies the above property.

\textbf{Question 2.} Does there exist a space with exactly two subsymmetric bases, which are not symmetric?

\end{document}